%% file: energy_final.tex
\documentclass{amsart}
\usepackage{amsmath}
  \usepackage{paralist}
  \usepackage{graphics} %% add this and next lines if pictures should be in esp format
  \usepackage{epsfig} %For pictures: screened artwork should be set up with an 85 or 100 line screen
 \usepackage[colorlinks=true]{hyperref}
\hypersetup{urlcolor=blue, citecolor=red}

\newtheorem{theorem}{Theorem}[section]

\newtheorem{lemma}[theorem]{Lemma}
\newtheorem{proposition}{Proposition}

\theoremstyle{definition}
\newtheorem{definition}[theorem]{Definition}
\newtheorem{remark}{Remark}

\newcommand{\dist}{\mathrm{dist}}

\def\LM#1{\hbox{\vrule width.2pt \vbox to#1pt{\vfill \hrule width#1pt height.2pt}}}
\def\LL{{\mathchoice {\>\LM7\>}{\>\LM7\>}{\,\LM5\,}{\,\LM{3.35}\,}}}
\def\restr{{\LL}}

\title[SIF for non-smooth fractures]
      {The Stress Intensity Factor for non-smooth fractures in antiplane elasticity.}

\author[Antonin Chambolle and Antoine Lemenant]{Antonin Chambolle
and
Antoine Lemenant}
\address{A. Chambolle: CMAP, Ecole Polytechnique, CNRS, France}
\email{antonin.chambolle@polytechnique.fr}
\address{A. Lemenant: LJLL, Universit\'e Paris-Diderot, CNRS, Paris, France}
\email{lemenant@ann.jussieu.fr}


\newcommand{\Hh}{\mathcal{H}^1}

\newcommand{\R}{\mathbb R}
\newcommand{\N}{\mathbb N}
\newcommand{\e}{\varepsilon}
\newcommand{\dv} {\mathrm{div}}

\begin{document}

%The abstract of your paper

% % Enter the first author's name and address:
% \centerline{\scshape Antonin Chambolle }
% \medskip
% {\footnotesize
% % please put the address of the first author
%  \centerline{CMAP, Ecole polytechnique, CNRS}
%    %\centerline{Other lines}
%    \centerline{Palaiseau, France}
% } % Do not forget to end the {\footnotesize by the sign }

% \medskip

% \centerline{\scshape Antoine Lemenant}
% \medskip
% {\footnotesize
%  % please put the address of the second  and third author
%  \centerline{ LJLL, Universit\'e Paris-Diderot, CNRS }
%  %\centerline{ Universit\'e Paris Diderot - Paris 7 }
%  %  \centerline{U.F.R de Math\'ematiques, Site Chevaleret}
%  %  \centerline{Case 7012, 75205 Paris Cedex 13, France}
%  \centerline{Paris, France}
% }
% \bigskip

% The name of the associate editor will be entered by an editorial staff
% "Communicated by the associate editor name" is not needed for special issue.
% \centerline{(Communicated by the associate editor name)}

\begin{abstract}
Motivated by some questions arising in the study of quasistatic growth in brittle fracture, we investigate the asymptotic behavior of the energy of the solution $u$ of a Neumann problem near a crack in dimension 2. We consider non smooth cracks $K$ that are merely closed and connected.
At any point of density 1/2 in $K$, we show that the blow-up limit of $u$ is
the usual ``cracktip'' function $\sqrt{r}\sin(\theta/2)$, with a well-defined coefficient
(the ``stress intensity factor'' or SIF).
The method relies on Bonnet's monotonicity formula \cite{b} together with $\Gamma$-convergence techniques. 
%We hope to lay the foundations in a nonsmooth context for some  works pursued in the past with smooth fractures only. 
\end{abstract}
\maketitle

%The title of your section 1
\section{Introduction}

According to Griffith's theory, the propagation of a brittle fracture in an elastic body is governed by the competition between the energy spent to produce a crack, proportional to its length, and the corresponding release of bulk energy. An energetic formulation of this idea is the core of variational models for crack propagation,
which were introduced by Francfort and Marigo in \cite{fm}
and are based on a Mumford-Shah-type~\cite{MumfordShah}  functional.

In this work, we will restrict ourselves to the case of \emph{anti-plane shear},
where the domain is a cylinder $\Omega \times \R$, with $\Omega \subset \R^2$, which is linearly elastic with Lam\'e coefficients $\lambda$ and $\mu$. Moreover we assume the crack  to be   vertically invariant, while the displacement is vertical only.
Under those assumptions, the problem reduces to a purely 2D, scalar problem. Extending
our result to (truely 2D) planar elasticity requires a finer knowledge of 
monotonicity formulas for the bilaplacian and is still out of reach, it is
the subject of future study.

Given a loading process $g:t \mapsto g(t) \in H^1(\Omega)$,
and assuming that $K(t)\subset \Omega$ (a closed set) is the fracture at time $t$,
the bulk energy at the time $t_0$ is given by
\begin{equation}
E(t_0):=\min_u  %%\frac{\mu}{2}
\int_{\Omega\setminus K(t_0)}(A\nabla u)\cdot \nabla u\;dx \,, \label{minE}
\end{equation}
where the minimum is taken among all functions $u \in H^1(\Omega \setminus K(t_0) ,\R)$ satisfying $u=g(t_0)$ on $\partial \Omega \setminus K(t_0)$, and the surface energy, for any fracture $K \supseteq K(t_0) $ is proportional to $\kappa\Hh(K)$, where $\Hh$ denotes the one dimensional Hausdorff measure and $\kappa$ is a constant which is known as the \textit{toughness} of the material.
Here the matrix $A$ which appears in the integral in~\eqref{minE} is
$(\mu/2)Id$, however in the paper we will also address the case of
more general matrices $A(x)$, which will be assumed to be uniformly elliptic
and spatially H\"older-continuous.

The proof of existence for a crack $K(t)$ satisfying the propagation criterions of brittle fracture as postulated by Francfort and Marigo in  \cite{fm}, was first proved by Dal Maso and Toader  \cite{dmt} in the simple  2D linearized anti-plane setting, then  extended in various directions by other authors \cite{cdens,dmft,fl,bg}.

In this paper we will freeze the ``time'' at a certain fixed value $t_0$,
and therefore do not really matter  exactly which model of existence we use.
We will only need to know that such fractures exist, as a main motivation
for our results. 

In the quasistatic model, the fracture $K(t)$ is in equilibrium at any time, which means that the total energy cannot be improved at time $t_0$ by extending the crack.
Precisely, for any closed set $K\supseteq K(t_0)$ such that $K(t_0)\cup K$ is connected, and for any  $u \in H^1(\Omega \setminus (K(t_0) \cup K))$ satisfying $u=g(t_0)$ on $\partial \Omega \setminus (K(t_0)\cup K)$, one must have that
\[
E(t_0)+\kappa\Hh(K(t_0)) \leq %%\frac{\mu}{2}
\int_{\Omega\setminus ( K(t_0)\cup K)}(A\nabla u)\cdot \nabla u\;dx + \kappa \Hh(K).
\]
This implies that  the propagation of the crack is totally dependent on the external force $g$, and a necessary condition for $K$ to propagate is that of  the first order limit of the bulk energy, namely
\begin{equation}
\limsup_{h\to 0^+}\frac{E(t_0+h)-E(t_0)}{h}, \label{enrelease}
\end{equation}
to be greater or equal to $\kappa$.
The limit in  \eqref{enrelease} can be interpreted as an \emph{energy release rate}
along the growing crack,
which is the central object of many recent works \cite{c,cfm,cfm1,K,lt}.  

In all the aforementioned papers, a strong regularity assumption  is made on the fracture $K(t)$:  it is assumed to be a segment near the tip in \cite{cfm,cfm1,K}; to our
knowledge the weakest assumption is the $C^{1,1}$ regularity in \cite{lt}.
The main reason for this is the precise knowledge of the asymptotic development of
the displacement $u$ near the tip of the crack, when it is straight.
Indeed the standard elliptic theory in polygonal domains (see e.g. Grisvard \cite{gris}) says that in a small ball $B(0,\varepsilon)$ (we assume the crack tip is the origin),
if $u$ denotes the minimizer for the problem \eqref{minE},
then there exists $\tilde u \in H^2(B(0,\varepsilon)\setminus K(t_0))$ such that
\begin{equation}
u = C\sqrt{r}\sin(\theta/2) + \tilde u, \label{dev}
\end{equation}
(in polar coordinates, assuming the crack is $\{\theta=\pm\pi\}$).
In fracture theory, the constant $C$ in front of the sinus is usually referred as the \emph{stress intensity factor} (SIF).
In \cite{lt}, G.~Lazzaroni and R.~Toader proved that
 \eqref{dev} is still true if $K(t_0)$ is a $C^{1,1}$ regular curve, up to a
change of coordinates, and they base their  study of the energy release rate upon 
this fact.

The main goal of this paper is to extend \eqref{dev} to fractures that are 
merely closed and connected sets, and asymptotic to a half-line at small scales.
(We will need the technical assumption that the Hausdorff density is  $1/2$ at the origin,
that is, the length in small balls is roughly the radius ---
which basically means that $K(t_0)$ admits a tangent, up to suitable rotations.)
Our main result is as follows:

\begin{theorem} \label{the} Assume that $K:=K(t_0)\subset \Omega \subset \R^2$
is closed and connected, and let $u$ be a solution for the minimizing problem in \eqref{minE} with some $\alpha$-h\"{o}lderian coefficients $A:\Omega \to \mathcal{S}^{2\times 2}$. Assume that $x_0\in K\cap \Omega$ is a point of density $1/2$, that is,
$$\limsup_{r \to 0}\frac{\mathcal{H}^1(K\cap B(x_0,r))}{2r}=\frac{1}{2}$$
and that $A(x_0)=Id$. Then the limit
\begin{equation}
\lim_{r\to 0}\frac{1}{r}\int_{B(x_0,r)\setminus K} (A \nabla u )\cdot \nabla u \; dx \label{limite}
\end{equation}
exists and is finite. Moreover denoting $C_0$ the value of this limit, considering $R_r$ a suitable family of rotations, and taking
$$g(r,\theta):=u(0)+\sqrt{\frac{2C_0}{\pi}} \sin(\theta/2),\quad (r,\theta) \in [0,1]\times [-\pi,\pi],$$
then the blow up sequence $u_r:=r^{-\frac{1}{2}}u(rR_r(x-x_0))$ converges to $g$ and $\nabla u_r$ converges to $\nabla g$ both strongly in $L^2(B(0,1))$ when $r \to 0$.
\end{theorem}

If $A(x_0)\not = Id$ we obtain a similar statement by applying the change of variable $x\mapsto \sqrt{A(x_0)} x$ (see Theorem \ref{prop2}). We also stress that a rigourous sense to the value of $u(0)$ has to be given, and this will be done in Lemma \ref{defu0}. Besides, the exact definition of the rotations $R_r$ will be given in Remark \ref{rotations}.

Theorem \ref{the} is a first step toward understanding the energy release rate for non-smooth fractures, and study qualitative properties of the crack path.  
It provides also the existence of a generalized stress intensity factor, that we can define as being the limit in \eqref{limite}, and which always exists without any regularity assumptions on $K(t_0)$ of that of being closed and connected (see Proposition \ref{prop1bisA}). 

Our main motivation is the study of brittle fracture, but of course Theorem \ref{the} contains  a general result about the regularity of solutions for a Neumann Problem in rough domains, that could be interesting for other purpose.

The proof of Theorem \ref{the} will be done in two main steps, presented in Section \ref{Bonnet} and Section \ref{blowup}, which will come just after some preliminaries (Section \ref{geom}).
The first step is to prove the existence of limit in \eqref{limite}. For this we will use the monotonicity argument of Bonnet \cite{b}, which was used to prove existence of blow up limits for the minimizers of the Mumford-Shah functional.
 We will adapt here the argument to more general energies as the one with coefficient $A(x)$, and also with a second member $f$.    Notice that when $f=0$ we need only $K$ to be closed and connected, whereas when $f\not = 0$ we need furthermore that $K$ is of finite length.  

The second step is to prove the convergence strongly in $L^2$ of the blow-up limit $u_r:=r^{-\frac{1}{2}}u(rR_r(x-x_0))$ and its gradient, to the function $\sqrt{r}\sin(\theta/2)$. This is the purpose of Theorem \ref{prop2}, and the existence of limit in \eqref{limite} is the first step, because it implies that $\nabla u_r$ is bounded in $L^2(B(0,1))$ which helps us to extract subsequences.

Notice that Bonnet \cite{b} already had a kind of blow-up convergence for $u_r$, analogue to ours in his paper on regularity for Mumford-Shah minimizers. The main difference with the result of Bonnet, is that here the set $K$ is any given set whereas for Bonnet, $K$ was a minimizer for the Mumford-Shah functional, which allowed him to modify it at his convenience to create competitors and prove some results on $u$. Here we cannot argue by the same way  and this brings some interesting technical difficulties in the proof of convergence of $u_r$.
%%%%%%%%%%%%%%%%%%%%%%%

\section{Preliminaries}\label{geom}

Let $\Omega \subset \R^2$, $K\subset \Omega$ be a closed and connected set satisfying   $\Hh(K)<+\infty$ (here $\Hh$ denotes the one dimensional Hausdorff measure),  $f \in L^\infty(\Omega)$, $\lambda \geq 0$ and $g \in H^1(\Omega )\cap L^\infty(\Omega)$. If $K$ and $K'$ are two closed sets of $\R^2$ we will denote the Hausdorff distance by 
$$d_H(K,K'):=\max\left(\sup_{x\in K}\dist(x,K'), \sup_{x\in K'}\dist(x,K)\right).$$

We also consider some $\alpha$-H\"older regular coefficients $A : x\mapsto A(x) \in \mathcal{S}^{2\times 2}$, uniformly bounded and uniformly coercive (with constant $\gamma$). We will use the following series of notations 
$$\|X\|_A:= ^t\!XAX = (AX)\cdot X=\|\sqrt{A}X\|_{Id}=\|\sqrt{A}X\|.$$
For simplicity we will assume without loss of generality that %% $\mu=2$ and
 $\kappa=1$. We consider a slight more general energy than the one in \eqref{minE} with a second member $f$, namely
\begin{equation}
F(u):=\int_{\Omega \setminus K}\|\nabla u\|_A^2 dx + \frac{1}{\lambda}\int_{\Omega}|\lambda u-f|^2 .\label{defF}
\end{equation}
We will also allow the case $\lambda=0$ and then we ask also $f=0$ and $F$ is simply
\begin{equation*}
F(u):=\int_{\Omega \setminus K}\|\nabla u\|_A^2 dx .
\end{equation*}

We consider  a minimizer $u$ for $F$ among all functions $v \in H^1(\Omega \setminus K)$ such that $v= g$ on $\partial \Omega$, i.e.  $u$ is a weak solution for the problem
\begin{equation} \label{problem1}
\left\{\begin{array}{cc}
\lambda u -\dv A\nabla u = f &\text{ in } \Omega \setminus K\\
  (A \nabla u)\cdot  \nu = 0 &\text{ on } K\\
u= g &\text{ on } \partial \Omega
\end{array}\right.
\end{equation}
 It is well known that such a minimizer exists and is unique (up to additional constant if necessary in connected components of $\Omega\setminus K$ when eventually $f=0$), which provides a week solution for the problem \eqref{problem1}.

We begin with some elementary geometrical facts.

\begin{proposition} Let $K \subset \mathbb{R}^2$ be a closed and connected set such that 
\begin{equation}
\limsup_{r \to 0} \frac{\mathcal{H}^1(K\cap B(x_0,r)) }{2r} = \frac{1}{2}. \label{density_condition}
\end{equation}
For all $r>0$ small enough, let $x_r$ be any chosen point in $K\cap \partial B(x_0,r)$. Then we have that 
\begin{equation}
\lim_{r\to 0}\frac{1}{r}d_H(K\cap B(x_0,r),[x_r,x_0])=0. \label{wtangent}
\end{equation}
\end{proposition}

\begin{proof} Since $K$ is closed, connected and not reduced to one point (because of \eqref{density_condition}) we have  that $K \cap \partial B(x_0,r)$ is nonempty for all $r$ small enough. Moreover since $K$ is connected, there exists a simple connected curve $\Gamma_r \subset K$ that starts from $x_0$ and hits $\partial B(x_0,r)$ for the first time at some point $y_r \in K \cap \partial B(x_0,r)$. Since $\Gamma_r$ is connected we have that $\mathcal{H}^1(\Gamma) \geq \mathcal{H}^1([y_r,x_0]) =r$ and using \eqref{density_condition} we get $\mathcal{H}^1(\Gamma_r) \leq r+ o(r)$. From the last two inequalities, since $\Gamma_r$ is a connected curve, it is then very classical using Pythagoras inequality to prove that
\begin{equation}
d_H(\Gamma_r, [y_r,x_0])=o(r). \label{haha}
\end{equation}
Indeed, let $z$ be the point in $\Gamma_r$ of maximal distance to $[y_r,x_0]$, and let $h$ be this distance. Now let  $w$ be a point at distance $h$ to $[y_r,x_0]$, whose orthogonal projection onto $[y_r,x_0]$ is exactly the middle of $[y_r,x_0]$. Then the triangle $(y_r,x_0,w)$ is isocel, and in particular minimizes the perimeter among all triangle of same basis and same height. Therefore,
$$2\sqrt{(r/2)^2+h^2}=|w-x_r|+|w-x_0|\leq |z-x_r|+|z-x_0| \leq \Hh(\Gamma_r) \leq r+o(r) $$ 
which implies that $h= o(r)$ and proves \eqref{haha}.

Now \eqref{density_condition} also implies that
$$\mathcal{H}^1(K\cap B(x_0,r) \setminus \Gamma_r) = o(r),$$
from which we deduce that 
$$\sup\{\dist(x, \Gamma) ; \; x \in K \cap B(x_0,r) \} = o(r)$$
which implies $d_H(K \cap B(x_0,r) , [y_r,x_0])=o(r)$. Finally \eqref{wtangent} follows from the fact that $\dist(x_r,y_r)=o(r)$ for any other point $x_r \in K \cap \partial B(x_0,r)$.
\end{proof}

\begin{remark} The density condition \eqref{density_condition} does not imply the existence of tangent at the origin. One of such example can be found in \cite[Remark 2.7.]{cfm}, as being a curve with oscillating tangent at the origin: $\exp(-t^2)(\cos(t)\bold{e}_1+\cos(t)\bold{e}_2)$, $t\in [0,+\infty]$. A further example is given by some infinite spirals, that turns infinitely many times around the origin but has finite length, and even density $1/2$ at the origin (thus is arbitrary close to a segment). To construct such an example one can consider the curve $\gamma : t \mapsto t e^ {i \theta(t)} $ where $\theta(t) \in \mathbb{R}$ satisfies $\lim_{t \to 0}\theta(t)=+\infty$  and $\lim_{t \to 0} t \theta'(t)=0$ (uniformly) like for instance $\theta(t)=\sqrt{\ln(-t)}$. Then if $K:=\gamma([0,1])$ we have 
$$\mathcal{H}^1(K\cap B(0,r))=\int_0^r |\gamma'(t)|dt =\int_0^r \sqrt{1+t^2\theta'(t)^2}dt=r+o(r)$$
as desired. 
\end{remark}
\begin{remark}[Definition of $R_r$] \label{rotations} As noticed in the preceding remark, the existence of tangent, i.e. the existence of a limit for the sequence of rescaled set $\frac{1}{r}(K\cap B(x_0,r)-x_0)$, is not always guaranteed by the density condition. On the other hand if $R_r$ denotes for each $r>0$, the rotation that maps $x_r$ on the negative part of the first axis, then $R_r(\frac{1}{r}(K\cap B(x_0,r)-x_0))$ converges to a segment. In the sequel, $R_r$ will always refer to this rotation. 
\end{remark}
\begin{remark}\label{subseq} There exists some connected sets such that $\frac{1}{r_n}K\cap B(0,r_n)$ converges to some radius  in $B(0,1)$ for some sequence $r_n \to 0$, and such that $\frac{1}{t_n}K\cap B(0,t_n)$ converges to a diameter for another sequence $t_n\to 0$. Such a set can be constructed as follows. Take a sequence $q_n \to 0$ such that 
\begin{equation}
q_{n+1}/q_n\ \longrightarrow\ 0
%% \frac{q_{n+1}}{q_n}\to 0
 \label{geometric}
\end{equation}
The idea relies on the observation that thanks to \eqref{geometric}, while looking at the scale of size $q_n$, that is, in the ball $B(0,q_n)$, all the piece of set contained in $B(0,q_{n+1})$ is negligible in terms of Hausdorff distance. Therefore we can build two subsequences, one at the scales $q_{2n}$, and the other
one at the scales $q_{2n+1}$, that will not be seen by each other. 
%%%%%%DESSIN
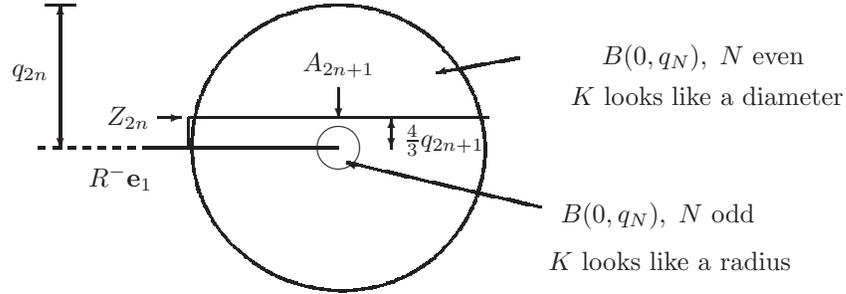
\begin{figure}[h]
\vspace{-16mm}
\begin{center}
\input{dessin1.tex}
\end{center}
\vspace{-8mm}
\caption{A crack tip with two different limits along different subsequences}\label{fig:strangetip}
\end{figure}
%%%%%%%%%% 

To do so, we consider the points $A_n:=(0,\frac{4}{3}q_{n})$ on the second axis of $\R^2$ and  we define $K$, as being $\R^-\times \{0\}$ union of all  horizontal diameters of $B(A_{2n+1},q_{2n})$, that are connected to the first axis by their left extremities. In other words,  denoting ${\bold e}_1=(1,0)$ and ${\bold e}_2=(0,1)$  the two unit canonical vectors of $\R^2$, 
$$K:= \R^-{\bold e}_1 \cup  \big(\bigcup_{n \in \N^*}\left(\R {\bold e}_1+A_{2n+1})\cap B(A_{2n+1},q_{2n} ) \right)\cup  \bigcup_{n\in \N^*} ([0,\frac{4}{3}q_{n+1}] {\bold e}_2 + Z_{2n}),$$
where $Z_{2n}$ is the left extremity point of the segment  ($\R e_1+A_{2n+1})\cap B(A_{2n+1},q_{2n} )$ (which is actually the horizontal diameter of $B(A_{2n+1},q_{2n})$), see Fig.~\ref{fig:strangetip}.

Then it is easy to see that, in the Hausdorff distance,
\begin{eqnarray}
 & & \frac{1}{q_{2n}}K\cap B(x,q_{2n}) \longrightarrow
 \R {\bold e}_1\cap B(0,1)\notag \\
  &\text{ and }&  \frac{1}{q_{2n+1}}(K\cap U(x,q_{2n+1}) \longrightarrow
 \R^- {\bold e}_1\cap B(0,1) 
\notag
\end{eqnarray}
as desired. 
\end{remark}
\begin{remark} Notice that a consequence of Theorem \ref{the} for the example  exhibited in  Remark \ref{subseq} is the following curious fact: even if $\frac{1}{r}K\cap B(x,r)$ has no limit when $r\to 0$, the limit of $\frac{1}{r}\int_{B(0,r)}\|\nabla u \|^2$ as $r\to 0$ exists thus has same value $C_0$ for any subsequences of $r$. Now, since $K$ has density $1/2$ along the odd subsequence  $r_n=q_{2n+1}$, applying the proof of Theorem \ref{the} for this subsequence we infer that  the limit of the blow up sequence $r_n^{-1/2}u(r_nx)$ converges to $\sqrt{2C_0r/\pi}\sin(\theta/2)$. But now regarding the limit in the even scales, $r_n=q_{2n}$, as $K$ is converging to a diameter, a similar proof as the one used to prove Theorem \ref{the} would imply that the blow up sequence is converging to the solution of a Neumann problem in a domain which is a ball, cut into two pieces by a diameter. This implies  $C_0=0$ (because of the decomposition of $u$ in spherical harmonics), so that actually returning to the   odd subsequence, for which  $K$ is converging to a radius,  we can conclude that $r_n^{-1/2}u(r_nx)$  must converge to $0$ as well.
\end{remark}

It is well known that any closed and connected set $K$ is arcwise connected, namely for any $x \not = y$  in  $K$ one can find an injective Lipschitz curve inside $K$ going from $x$ to $y$ (see e.g. \cite[Proposition 30.14]{d}). This allows us to talk about geodesic curve inside $K$, that connects $x$ to $y$, which stands to be the curve with that property which support has minimal length. 

\begin{definition} \label{chordarc}We say that $K$ is locally-chord-arc at $x_0$ if there exists a constant $C$ and a radius $r_0$ such that for every $r \leq r_0$  and for any couple of points $y$ and $z$ lying on $K \cap \partial B(x_0,r)$ the geodesic curve inside $K$ connecting $y$ and $z$ has length less than $Cr$.
\end{definition}

\begin{proposition} Let $K \subset \mathbb{R}^2$ be a closed and connected set satisfying the density condition
\begin{equation}
\limsup_{r \to 0} \frac{1}{2r}\mathcal{H}^1(K\cap B(x_0,r))  = \frac{1}{2}. \label{density_condition2}
\end{equation}
Then $K$ is locally-chord-arc at $x_0$.
\end{proposition}

\begin{proof} The density condition \eqref{density_condition2} together with  the fact that $K$ is closed and connected guarantees  that  $K$ is non reduced to one point, contains $x_0$, and that $\partial B(x_0,r)\cap K$ is nonempty for $r$ small enough. Let $r_0>0$ be one of this radius small enough such that moreover 
\begin{equation}
\mathcal{H}^1(K \cap B(x_0,r))\leq (1+\frac{1}{10})r \quad \forall r\leq 3r_0. \label{density1}
\end{equation}
Let now $y$ and $z$ be two points in $K \cap \partial B(x_0,r)$ for any $r\leq r_0$ and let $\Gamma \subset K$ be the geodesic curve connecting $y$ and $z$. Then $\Gamma$ is injective (by definition since it is a geodesic) and in addition we claim that $\Gamma \subset B(x_0,3r)$. Indeed, otherwise there would be a point $x \in \Gamma \setminus B(x_0,3r)$ which would imply $\mathcal{H}^1(\Gamma \cap B(x_0,3r))\geq 4r$ (because $y$ and $z$ are lying on $\partial B(x_0,r)$) and this contradicts \eqref{density1}.  But now that $\Gamma \subset B(x_0,3r)$, condition \eqref{density1} again implies that $\mathcal{H}^1(\Gamma)\leq \Hh(K\cap B(x_0,3r))\leq 4r$ which proves the proposition.
\end{proof}

In the sequel we will need to know that a minimizer of $F$ is bounded.
\begin{proposition}  Let $K$ be closed and connected, $u$ be a minimizer for the functional $F$ defined in \eqref{defF} with $f \in L^{\infty}(\Omega)$ and $\lambda > 0$. Then 
$$\|u\|_\infty \leq \frac{1}{\min(1,\lambda)}\max(\|f\|_\infty,\|g\|_\infty).$$
\end{proposition}
\begin{proof} It suffice to fix $M:=(\min(1,\lambda))^{-1}\max(\|f\|_\infty,\|g\|_\infty)$ and notice that the function 
$$w:=\max(-M,\min(u,M))$$
is a competitor for $u$ and has less energy. By uniqueness of the minimizer we deduce that $u=w$. 
\end{proof}

\section{Bonnet's monotonicity Lemma and variants}
\label{Bonnet} In this section we prove the existence of the limit 
$$\lim_{r \to 0}\frac{1}{r}\int_{\sqrt{A(0)}B(x,r)}\|\nabla u \|_A^2 dx,$$
for any $x \in \Omega$ when $u$ is a minimizer of $F$. Of course when $x \in \Omega \setminus K$ this is clear by the interior regularity of solution for the Problem \ref{problem1}, and the value of the limit in this case is zero. Therefore it is enough to consider a point $x \in K$.

The case of harmonic functions is slightly simpler than the general case, and need no further assumptions on $K$ than being just closed and connected. This direclty comes from \cite{b} and \cite{d} but we will recall the proof in Subsection \ref{harmonic}, that follows a lot the approach of G. David \cite[Section 47]{d}. Then we will consider the case of a non zero second member $f$ but still with the classical Laplace operator, and finally in a third section we will adapt all the proofs to more general second order operator of divergence form.

We begin with some technical tools.

\subsection{Technical tools}

We will need the following 2 versions of the Gauss-Green formula.

\begin{lemma}[Integration by parts, first version]\label{ipp} Let $K$ be closed and connected, $u$ be a minimizer for the functional $F$ defined in \eqref{defF}. Then for any $x \in \Omega$ and for a.e.  $r$ such that $B(x,r)\subset \Omega$ it holds
$$\int_{B(x,r)\setminus K}\|\nabla u \|_A\; dy = \int_{B(x,r)\setminus K}(f-\lambda u)u \;dy + \int_{\partial B(x,r)\setminus K} u (A\nabla u)\cdot \nu \;d\Hh.$$
\end{lemma}

\begin{proof} If $u$ is a minimizer, then comparing the energy of $u$ with the one of $u+t \varphi$ and using a standard variational argument yields that
\begin{equation}
\int_{\Omega \setminus K} (A \nabla u ) \cdot \nabla \varphi \,dx
 = \int_{\Omega} (\lambda u-f)\varphi\,dx
\label{ipp2}
\end{equation}
must hold  for any  function $\varphi \in H^1(\Omega \setminus K)$ compactly supported inside $\Omega$. Let us choose $\varphi$ to be equal to $\psi_\varepsilon u(x)$, where $\psi_\varepsilon(x) =g_\varepsilon(\|x\|)$ is radial, and $g_\varepsilon$ is equal to $1$ on $[0,(1-\varepsilon)r]$, equal to $0$ on $[r+\varepsilon,+\infty[$ and linear on $[(1-\varepsilon)r,(1+\varepsilon)r]$. Applying \eqref{ipp2} with $\varphi= \psi_\varepsilon u$ gives
\begin{equation}
\int_{\Omega \setminus K} ( A \nabla u ) \cdot  u\nabla \psi_\varepsilon   +
\int_{\Omega \setminus K} (A \nabla u ) \cdot   \psi_\varepsilon \nabla u = \int_{\Omega} (\lambda u-f)\psi_\varepsilon u.\label{conv}
\end{equation}
It is clear that $\psi_\varepsilon$ converges to ${\bf 1}_{B(x,r)}$ strongly in $L^2(\Omega)$, which gives the desired convergence for the second term and last term in \eqref{conv}. Now for the first term, we notice that $\psi_\varepsilon$ is Lipschitz and its derivative is equal a.e. to $\frac{x}{2\varepsilon \|x\|} {\bf 1}_{B(x,r+\varepsilon)\setminus B(x,(1-\varepsilon)r)} $ so that
\begin{eqnarray}
\int_{\Omega \setminus K} ( A \nabla u ) \cdot  u\nabla \psi_\varepsilon   &=& 
\frac{1}{2\varepsilon}\int_{(B(x,r+\varepsilon)\setminus B(x,(1-\varepsilon)r))\setminus K} (A \nabla u) \cdot  u \frac{x}{\|x\|}   \notag
\end{eqnarray}
which converges to $\int_{\partial B(x,r)\setminus K} (A \nabla u)\cdot  u \;\nu\;   d\Hh$ for a.e. $r$ by Lebesgue's differentiation theorem applied to the $L^1$ function $r \mapsto \int_{\partial B(x,r) \setminus K} ( A \nabla u )\cdot   \nu  \; d\Hh$.
\end{proof}

The first part of the next Lemma comes from a  topological  argument in \cite{d} (see page 299).

\begin{lemma}[Integration by parts, second version]\label{ippV2} Let $K\subset \Omega$ be closed and connected, $x\in K$ and $r_0>0$ be such that $B(x,r_0)\subset \Omega$.  For all $r \in (0,r_0)$ we decompose 
$\partial B(x,r) \setminus K= \bigcup_{j \in J(r)} I_j(r)$
where $I_j(r)$ are disjoints arcs of circles. Then for each $j \in J(r)$ there exists a connected component $U_j(r)$ of $\Omega \setminus (I_j(r)\cup K)$ such that 
$$\partial U_j(r) \setminus K = I_j(r).$$
Moreover if $u$ is a minimizer for the functional $F$ defined in \eqref{defF}, then for a.e. $r \in (0,r_0)$ and for every $j \in I_j(r)$ we have 
\begin{equation}
\int_{I_j(r)}  (A\nabla u) \cdot \nu \; d\Hh = \int_{U_j} (\lambda u - f) dx,  \label{ipp_formule}
\end{equation}
where $\nu$ is the inward normal vector  in $U_j$, i.e. pointing inside $U_j$. 
\end{lemma}
\begin{proof} By assumption $K$ is closed, so that $\partial B(0,r)\setminus K$ is a relatively open set in $\partial B(0,r)$ which is one dimensional. Therefore we can decompose 
$\partial B(0,r)\setminus K$ as a union of  arc of circles as in the statement of the Lemma, namely 
$$\partial B(0,r)\setminus K = \bigcup_{j \in J} I_j.$$
(we avoid the dependance in $r$ to lighten the notations). Let  us denote by $U^+_j$  the connected component of $\Omega \setminus (K \cup I_j)$ containing the points of $B(0,r)\setminus K$ very close to $I_j$, and similarly $U^-_j$ is the one containing the points of $\Omega \setminus (K\cup B(0,r))$ very close to $I_j$. Then there is one between $U^\pm_j$, that we will denote by $U_j$, which satisfies
\begin{equation}
\partial U_j \setminus K = I_j. \label{topolog}
\end{equation}
The proof of  \eqref{topolog} relies on the connectedness of $K$  (see \cite{d} page 299 and 300 for details:  in our case the connectedness of $K$ implies the topological assumption denoted by  (8) in \cite{d} that is used to prove \eqref{topolog} (which is actually (14) in \cite{d})).

Then we want to prove \eqref{ipp_formule} by an argument similar to Lemma \ref{ipp} applied in $U_j$. For this purpose we consider as before a radial function but we need to separate two cases: if $U_j\subset B(0,r)$ then we take the same function $\psi_\varepsilon(x) =g_\varepsilon(\|x\|)$ where $g_\varepsilon$ is equal to $1$ on $[0,(1-\varepsilon)r]$, equal to $0$ on $[r+\varepsilon,+\infty[$ and linear on $[(1-\varepsilon)r,(1+\varepsilon)r]$. Now if $U_j \subset \Omega \setminus B(0,r)$ we define $\psi_\varepsilon := 1- \psi_\varepsilon$.

Then we take as a competitor for $u$ the function   $\varphi = {\bf 1}_{\hat U_j}\psi_\varepsilon $, where  $\hat U_j$ is the connected component of $\Omega \setminus K$ containing $U_j$. Notice that this is an admissible choice, namely $\varphi \in H^{1}(\Omega \setminus K)$ and $\varphi=0$ on $\partial \Omega$.

Applying \eqref{ipp} with $\varphi= {\bf 1}_{\hat U_j}\psi_\varepsilon $ gives
\begin{equation}
\int_{\hat U_j} ( A \nabla u ) \cdot  \nabla \psi_\varepsilon  = \int_{\hat U_j} (\lambda u-f)\psi_\varepsilon . \label{formipp}
\end{equation}
As in the proof of Lemma \ref{ipp}, it is clear that $\psi_\varepsilon$ converges to ${\bf 1}_{U_j}$ strongly in $L^2(\Omega)$, which gives the desired convergence for the right hand side term  in \eqref{formipp}. Now for the left hand side term, we use as before that $\psi_\varepsilon$ is Lipschitz and its derivative is equal a.e. to $\pm \frac{x}{2\varepsilon \|x\|} {\bf 1}_{B(x,r+\varepsilon)\setminus B(x,(1-\varepsilon)r)} $ (with the correct sign depending on which side of $I_j$ lies $U_j$) so that
\begin{eqnarray}
\int_{\Omega \setminus K} (A \nabla u)\cdot   \nabla \psi_\varepsilon  &=& 
\pm \frac{1}{2\varepsilon}\int_{\hat U_j \cap (B(x,r+\varepsilon)\setminus B(x,(1-\varepsilon)r))} ( A \nabla u)\cdot   \frac{x}{\|x\|}   \notag
\end{eqnarray}
which converges to $\int_{I_j} (A \nabla u )\cdot  \nu   d\Hh$ for a.e. $r$ by Lebesgue's differentiation theorem applied to the $L^1$ function $r \mapsto \int_{\partial B(x,r) \cap \hat U_j} ( A \nabla u )\cdot   \nu  \; d\Hh$.
\end{proof}

%%%%%%%%%%%%%%%%%%%%%%%%%%%%%%%%%%%%%%%%%

\subsection{Monotonicity for harmonic functions}
\label{harmonic}
So we arrive now to the first monotonicity result. The following proposition is one of the key points in Bonnet's proof of the classification of global minimizers for the Mumford-Shah functional \cite{b} (see also Section 47 of Guy David's book \cite{d} for a more detailled proof with slightly weaker assumptions than \cite{b}). The same argument was also used in  Lemma 2.2. of \cite{L1} to prove a monotonicity result for the energy of a harmonic function in the complement of minimal cones in $\R^3$, but the rate of decay obtained by this method is sharp only in dimension 2. Notice also that a similar argument with the elastic energy (i.e. $L^2$ norm of the symmetric gradient) of a vectorial function $u : \Omega \to \R^2$ seems not to be working. Notice that in \cite{d} the assumption $\Hh(K)<\infty$ is needed whereas $K$ is not necessarily connected. Here we do not suppose $\Hh(K)<\infty$ but we ask $K$ to be connected which is a stronger topological assumption but weaker regularity assumption than \cite{d}.

\begin{proposition}[Monotonicity of Energy, the harmonic case]\label{prop1} \cite{b,d} Let $K$ be a closed and connected set and let $u$ be a solution for the problem \eqref{problem1} with $A=Id$, $f=0$ and $\lambda=0$ (therefore $u$ is harmonic in $\Omega \setminus K$). For any point $x_0 \in K$ we denote
 $$E(r):=\int_{B(x_0,r)\setminus K}\|\nabla u \|^2 dx.$$
Then $r\mapsto {E(r)}/r$ is an increasing function of $r$ on $(0,r_0)$. As a consequence,  the limit $\lim_{r\to 0} E(r)/r$ exists and is finite.
\end{proposition}

\begin{proof}  Let us rewrite here the proof contained in  \cite{d} and \cite{b} since we want to generalize it just after. We assume without loss of generality that $x_0$ is the origin. Firstly, it is easy to show that $E$ admits a derivative a.e. and
\begin{equation}
E'(r):=\int_{\partial B(0,r)\setminus K} \|\nabla u\|^2 dx. \label{derive}
\end{equation}
In addition $E$  is absolutely continuous (see \cite{d}). Therefore, to prove the monotonicity of $r\mapsto E(r)/r$, it is enough to prove the inequality
\begin{equation}
E(r)\leq r E'(r) \quad \quad \text{ for a.e. } r\leq r_0, \label{amontrer}
\end{equation}
because this implies $\big(E(r)/r)'\geq 0$ a.e.

We will need  Wirtinger's inequality  (see e.g. page 301 of \cite{d}), i.e. for any arc of circle $I_r \subset \partial B(0,r)$ and for  $g \in W^{1,2}(I_r)$ we have
\begin{equation}
\int_{I_r} |g-m_g|^2d\Hh \leq  \left(\frac{|I_r|}{\pi}\right)^2  \int_{I_r} |g'|^2d\Hh \label{Wirtinger}
\end{equation}
where $m_g$ is the average of $g$ on $I_r$ and $g'$ is the tangential derivative on the circle. The constant here is optimal, and is achieved for the unit circle by the function $\sin(\theta/2)$ on the arc of circle $]-\pi,\pi[$.

The first Gauss-Green formula (i.e. Lemma \ref{ipp}) applied in $B(0,r)$ yields, for a.e. radius $r$,
\begin{equation}
\int_{B(0,r)\setminus K}\|\nabla u \|^2 dx = \int_{\partial B(0,r)\setminus K} u \frac{\partial u}{\partial \nu} d \mathcal{H}^1. \label{ineq1}
\end{equation}
Now $K$ is closed, so that $\partial B(0,r)\setminus K$ is a relatively open set in $\partial B(0,r)$ which is one dimensional. Therefore we can decompose 
$\partial B(0,r)\setminus K$ as a union of  arc of circles as in Lemma \ref{ippV2}, namely 
$$\partial B(0,r)\setminus K = \bigcup_{j \in J} I_j.$$
Next, we apply Lemma \ref{ippV2} to obtain for each of those arcs $I_j$,
\begin{equation}
\int_{I_j}\frac{\partial u}{\partial \nu} d\mathcal{H}^1=0. \label{topol}
\end{equation}

Denoting by $m_j$ the average of $u$ on $I_j$ we deduce that
\begin{equation}
\int_{I_j} u \frac{\partial u}{\partial \nu} d\Hh = \int_{I_j} (u-m_j) \frac{\partial u}{\partial \nu} d\Hh  \label{etape}
\end{equation}
Returning to \eqref{ineq1} and plugging \eqref{etape} we get 
\begin{equation}
\int_{B(0,r)\setminus K}\|\nabla u \|^2 dx \leq \sum_{j\in J}\int_{I_j} |u-m_j| \Big|\frac{\partial u}{\partial \nu} \Big|d \Hh. \label{etape2}
\end{equation}
 Then by use of Cauchy-Schwarz inequality and  $ab \leq \frac{1}{2\epsilon}a^2+ \frac{\epsilon}{2}b^2$ we can write
\begin{eqnarray}
\int_{I_j} |u-m_j| \Big|\frac{\partial u}{\partial \nu}\Big| d\Hh &\leq& \left(\int_{I_j} |u-m_j|^2\right)^{\frac{1}{2}} \left(\int_{I_j}\left(\frac{\partial u}{\partial \nu} \right)^2  \right)^{\frac{1}{2}} \notag \\
&\leq & \frac{1}{2\epsilon}\int_{I_j} |u-m_j|^2 +
\frac{\epsilon}{2}\int_{I_j} \left(\frac{\partial u}{\partial \nu} \right)^2. \notag
 \end{eqnarray}
Using Wirtinger inequality and setting $\epsilon=2r$ we deduce that
\begin{eqnarray}
\int_{I_j} |u-m_j| \Big|\frac{\partial u}{\partial \nu}\Big|d\Hh  &\leq&
\frac{4r^2}{2\epsilon}\int_{I_j}\left(\frac{\partial u}{\partial \tau}\right)^2 +
 \frac{\epsilon}{2}\int_{I_j} \left(\frac{\partial u}{\partial \nu} \right)^2  \notag \\
 &\leq & r\int_{I_j} \left(\frac{\partial u}{\partial \tau}\right)^2 +
 r\int_{I_j} \left(\frac{\partial u}{\partial \nu} \right)^2  \notag  \\
 &=& r\int_{I_j} \|\nabla u\|^2 . \label{CS}
 \end{eqnarray}
Finally summing over $j$,  using \eqref{etape2} and \eqref{derive} we get \eqref{amontrer}
and the proposition is proved.
\end{proof}

\subsection{Monotonicity  with a second member $f$}

Now we start to prove some variants of Bonnet's monotonicity Lemma. If $f$ is non zero, then we obtain a similar result but  we need further assumptions on $K$ to be of finite length  and locally-chord-arc.
\begin{proposition}[Monotonicity, with a second member]\label{prop11} Let $u$ be a solution for the problem \eqref{problem1} with $A=Id$, $\lambda >0$ and $f,g \in L^\infty$, and assume that $K$ is a closed and connected set of finite length. For any point $x_0 \in K$ we denote
 $$E(r):=\int_{B(x_0,r)\setminus K}\|\nabla u \|^2 dx,$$
 and we denote by  $N(r)\in [0,+\infty]$  the number of points of $K \cap \partial B(0,r)$. We assume in addition that $K$ is locally-chord-arc at point $x_0$.  Then there exists a radius $r_0$ and a constant $C$ depending only on $\|f\|_\infty$, $\|g\|_\infty$ and the locally-chord-arc constant of $K$ such that
$$r\mapsto \frac{E(r)}{r}+CP(r)$$
is an increasing function of $r$ on $(0,r_0)$, where $P(r)$ is a primitive of $N(r)$. As a consequence,  the limit $\lim_{r\to 0} E(r)/r$ exists and is finite.
\end{proposition}

\begin{proof}  The proof is similar to the one of Proposition \ref{prop1}. We want to prove that the second member $f$ is just a perturbation under control which does not affect the limit of $E(r)/r$. Precisely, this time we will prove  the inequality
\begin{equation}
E(r)\leq r E'(r) + CN(r)r^2 \quad \quad \text{ for a.e. } r\leq r_0, \label{amontrerbis}
\end{equation}
with $N(r)\in L^1([0,r_0])$. This implies that  $\frac{d}{dr}(CP(r)+E(r)/r) \geq 0$ thus $r\mapsto E(r)/r+CP(r)$ is increasing  and this is enough to prove the Proposition because the limit of $P(r)$ exists at $0$.

We assume $x_0=0$. Observe that since $K$ has a finite length, we know that $\sharp K \cap \partial B(0,r)$ is finite for  a.e. $r \in (0,r_0)$. Actually we will need to know a bit more. If $N(r)$ denotes the number of points of $K \cap B(0,r)$, by \cite[Lemma 26.1.]{d} we know that $N$ is borel mesurable on $(0,r_0)$ and that
\begin{equation}
\int_0^{t} N(s) ds \leq \Hh(K\cap B(0,t)). \label{coaire}
\end{equation}
This will be needed later. For now, take  a radius $r$ a.e. in $(0,r_0)$ such that $N(r)<+\infty$  and decompose $S_r:=\partial B(0,r)\setminus K$ into a finite number of arcs of circle denoted $I_j$, for $j=1..N(r)$. Moreover since $K$ is closed and connected, for each $j$ there exists a geodesic curve $F_j\subset K$ connecting the two endpoints of $I_j$. We denote $D_j$ the domain delimited by $I_j$ and $F_j$. Since $K$ is locally-chord-arc at the origin we infer that
$|D_j|\leq C r^2$.
Notice also that $D_j$ corresponds to the set $U_j$ of Lemma \ref{ippV2}.

The Gauss-Green formula (Lemma \ref{ipp}) applied  in $B(0,r)$ yields
\begin{equation}
\int_{B(0,r)\setminus K}\|\nabla u \|^2 dx = \int_{B(0,r)\setminus K} (f-\lambda u)\,u\,dx + \sum_{i=1}^{N(r)}\int_{I_j} u \frac{\partial u}{\partial \nu} d\Hh, \label{ineq1bis}
\end{equation}
and applied in $D_j$ (Lemma \ref{ippV2}) gives
$$ \int_{I_j}  \frac{\partial u}{\partial \nu} d\Hh =  \pm \int_{D_j} (f-\lambda u) dx,$$
the sign depending on the relative position of $D_j$ with respect to $\partial B(0,r)$. Denoting by $m_j$ the average of $u$ on $I_j$ we deduce that
\begin{equation}
\int_{I_j} u \frac{\partial u}{\partial \nu} d\Hh = \int_{I_j} (u-m_j) \frac{\partial u}{\partial \nu} d\Hh \pm  \int_{D_j} m_j\,(f-\lambda u)\,dx. \label{etapebis}
\end{equation}
Now since $u$ is bounded it comes $|m_j|\leq C$, and we also have $\sum_{j=1}^{N(r)}|D_j|\leq C N(r) r^2$. Moreover $f$ is also bounded thus returning to \eqref{ineq1bis} and plugging \eqref{etapebis} we get
\begin{equation}
\int_{B(0,r)\setminus K}\|\nabla u \|^2 dx \leq CN(r)r^2 + \sum_{j=1}^{N(r)}\int_{I_j} |u-m_j| \Big|\frac{\partial u}{\partial \nu} \Big|d\Hh. \label{etape2bis}
\end{equation}
Then the same computations as for proving \eqref{CS} (i.e. using Cauchy-Schwarz inequality,  $ab \leq \frac{1}{2\epsilon}a^2+ \frac{\epsilon}{2}b^2$ and Wirtinger) we obtain
\begin{equation}
\int_{I_j} |u-m_j| \Big|\frac{\partial u}{\partial \nu} \Big|d\Hh   \leq r\int_{I_j} \|\nabla u \|^2 dx\;, \notag
 \end{equation}
and after summing over $j$,  \eqref{amontrerbis} is proved, as claimed.
\end{proof}
%%%%%%%%%%%%%%%%%%%%%%%%%%%%%%%

\subsection{The case of more a general operator}

We consider  now the general case with $\alpha$-H\"older regular coefficients $A : x\mapsto A(x) \in \mathcal{S}^{2\times 2}$, uniformly bounded and $\gamma$-coercive with $\gamma >0$. For any $x \in \Omega$ and $r>0$ we define the ellipsoid 
$$B_A(x,r):= \sqrt{A(x)}(B(x,r)).$$ 

\begin{proposition}[Change of variable]\label{chang} Let $u$ be a solution for Problem \ref{problem1} in $\Omega$, that we assume to contain the origin, and fix $A_0:=A(0)$.   Then  $u\circ\sqrt{A_0}$ is the solution of a similar problem in $(\sqrt{A_0})^{-1}(\Omega\setminus K)$ with coefficient $\tilde{A}:=(\sqrt{A_0})^{-1}\circ A \circ (\sqrt{A_0})^{-1}$ instead of $A$. In particular $\tilde{A}(0)=Id$, and 
\begin{equation}
\int_{B(0,r)}\|\nabla v\|_{\tilde A}^2dx = \int_{B_A(0,r))} \|\nabla u\|_A^2 \det{(\sqrt{A_0})^{-1}}dx \label{eq_energy_ch}
\end{equation}
\end{proposition}
\begin{proof} Let $u$ be a solution for Problem \ref{problem1}, and consider $v:=u\circ(\sqrt{A_0})$. Then since $\sqrt{A_0}$ is symmetric we have   $\nabla v(y) =  (\sqrt{A_0}) \nabla u ((\sqrt{A_0})(y))$ and  
\begin{eqnarray}
\int_{\Omega} \|\nabla u\|_A^2 dx &=& \int_{\Omega} \|\sqrt{A}\nabla u \|^2 dx \notag \\
&=& \int_{\sqrt{A_0}^{-1}(\Omega)} \|\sqrt{A}\nabla u \circ \sqrt{A_0}\|^2 \det(\sqrt{A_0}) dx \notag \\
&=&  \int_{\sqrt{A_0}^{-1}(\Omega)} \|\sqrt{A}\sqrt{A_0}^{-1}\nabla v\|^2 \det(\sqrt{A_0}) dx \notag \\
&=&  \int_{\sqrt{A_0}^{-1}(\Omega)} \|\nabla v\|^2_{\tilde A} \det(\sqrt{A_0}) dx \label{change1} 
\end{eqnarray}
with $\tilde{A}:=(\sqrt{A_0})^{-1}\circ A \circ (\sqrt{A_0})^{-1}$. Therefore if $u$ is a minimizer for the functional $F$ defined in \eqref{defF}, then $v$ must be a minimizer for the functional
\begin{equation}
\tilde {F}(v):=\int_{\tilde{\Omega}}\|\nabla v\|_{\tilde A}^2 dx + \frac{1}{\lambda}\int_{\tilde{\Omega}}|\lambda v- \tilde{f}  |^2 \label{defFtilde},
\end{equation}
with $\tilde{\Omega}:=\sqrt{A_0}^{-1}(\Omega \setminus K)$ and $\tilde f:=f\circ\sqrt{A_0}$. Finally, the same change of variable as the one used for \eqref{change1} proves \eqref{eq_energy_ch}, which completes the  proof of the proposition.
\end{proof}

Here is now  the analogue of Proposition \ref{prop1}.
\begin{proposition}[Monotonicity of energy for general coefficients]\label{prop1bisA} Assume that $K$ is a closed and connected set.  Let $u$ be a solution for the problem \eqref{problem1} with some $\alpha$-H\"older regular coefficients $A$, and with $\lambda =0$ and $f=0$. For any point $x_0 \in K$ we denote
 $$E(r):=\int_{B_A(x_0,r)\setminus K}\|\nabla u \|_A^2 dx.$$
Then the function
$$r \mapsto \frac{E(r)}{r}(1+Cr^{\frac{\alpha}{2}})^{\frac{2}{\alpha}}$$
is nondecreasing. As a consequence, the limit $\lim_{r\to 0}(E(r)/r)$ exists and is finite. 
  \end{proposition}

\begin{proof} We will use a third time a variation of Bonnet's monotonicity Lemma, i.e. we will follow again the proof of Proposition \ref{prop1}. Let us assume without loss of generality that $x_0$ is the origin. First of all,  up to the change of coordinates $x \mapsto \sqrt{A(0)}x$ and thank to Proposition \ref{chang} we can assume without loss of generality that $A(0)=Id$. In this case $B_A(x,r)=B_{Id}(x,r)=B(x,r)$.

The Gauss-Green formula (Lemma \ref{ipp}) applied in $B(0,r)$ yields
\begin{equation}
\int_{B(0,r)\setminus K}\|\nabla u \|_A^2 dx =  \sum_{j} \int_{I_j} u  (A\nabla u) \cdot \nu \; d\Hh, \label{ineq1r}
\end{equation}
where $\partial B(0,r)\setminus K = \cup_j I_j$. On the other hand Lemma \ref{ippV2} gives for each $j$,
$$ \int_{I_j}(A\nabla u) \cdot \nu \; dx = 0.$$
Denoting by $m_j$ the average of $u$ on $I_j$ we deduce that
\begin{equation}
\int_{I_j} u (A\nabla u)  \cdot \nu \; d\Hh = \int_{I_j} (u-m_j) (A\nabla u) \cdot \nu \; d\Hh . \label{etaper_bis}
\end{equation}
Thus
\begin{equation}
\int_{B(0,r)\setminus K}\|\nabla u \|^2_A dx \leq  \sum_{j=1}^N\int_{I_j} |u-m_j| | (A\nabla u)  \cdot \nu |\;d\Hh. \label{etape2ter}
\end{equation}
Then by use of Cauchy-Schwarz inequality,  $ab \leq \frac{1}{4r}a^2+ rb^2$, and Wirtinger we can write
\begin{eqnarray}
\int_{I_j} |u-m_j| |(A \nabla u) \cdot \nu| d\Hh &\leq& \left(\int_{I_j} |u-m_j|^2\right)^{\frac{1}{2}} \left(\int_{I_j} |(A \nabla u) \cdot  \nu |^2  \right)^{\frac{1}{2}} \notag \\
&\leq & \frac{1}{4r}\int_{I_j} |u-m_j|^2 +
r\int_{I_j} | (A\nabla u) \cdot \nu |^2 \notag \\
 &\leq&
r\int_{I_j} |\nabla u \cdot \tau|^2 +
r\int_{I_j} |(A\nabla u )\cdot \nu |^2.  \label{etape3}
 \end{eqnarray}
Now we want to recover the full norm $\|\nabla u\|_A$ from the partial norms $|\nabla u \cdot \tau|$ and $|(A\nabla u)\cdot \nu|$. For this purpose we write 
$$|\nabla u \cdot \tau |^2= |(A \nabla u) \cdot  \tau|^2+|(Id-A)\nabla u\cdot \tau|^2 + [2(A \nabla u) \cdot \tau][ ((Id-A)\nabla u)\cdot \tau],$$ 
and we notice that, by H\"older regularity of $A$  and $\gamma$-coerciveness we have (the constant $C$ can vary from line to line)
\begin{eqnarray}
|(Id-A)\nabla u\cdot \tau|^2 &\leq& \|(Id-A)\nabla u\|^2  \notag \\
&\leq & \|Id-A\|_{\mathcal{L}(\R^2,\R^2)}^2 \|\nabla u\|^2 \notag \\
 &\leq& Cr^{2\alpha}\|\nabla u\|^2 \notag \\
 &\leq & \gamma C r^{2\alpha}\|\nabla u\|^2_A \;,  \label{ineqnorm}
 \end{eqnarray}
 and
 \begin{eqnarray}
 2|A \nabla u \cdot  \tau (Id-A)\nabla u\cdot\tau| &=& 2|A \nabla u \cdot \tau|  |(Id-A)\nabla u \cdot \tau| \notag \\
 &\leq& 2 \|A \nabla u \|    Cr^{\alpha}\|\nabla u\|_A \notag \\
 &\leq & C r^{\alpha} \|A\|_{\infty} \|\nabla u\|_A^2 \notag.
 \end{eqnarray}
  Therefore summing over $j$ and putting all the estimates together we have proved that for $r$ small enough,
 \begin{eqnarray}
 \int_{B(0,r)\setminus K}\!\!
\|\nabla u \|_A^2 dx &\leq&
r\int_{\partial B(0,r)}\!\!
(|A\nabla u \cdot \tau |^2+Cr^{\alpha}\|\nabla u  \|_A^2)
+r\int_{\partial B(0,r)}\!\!
|(A\nabla u)\cdot \nu |^2 \notag \\
 &= & r\int_{\partial B(0,r)}\!\!
 \|A\nabla u\|^2 
+Cr^{1+\alpha}\int_{\partial B(0,r)}\!\!
 \|\nabla u  \|_A^2 . \label{inequ_c}
 \end{eqnarray}
  By H\"older regularity of $A$ we infer that 
 $$\|\sqrt{A}\|_{\mathcal{L}(\R^2,\R^2)}^2\leq 1+Cr^{\alpha/2},$$
which implies
\begin{eqnarray}
 \int_{\partial B(0,r)}\!\!
\|A\nabla u\|^2&\leq& \int_{\partial B(0,r)}\!\!
\|\sqrt{A}\|^2 \|\sqrt{A}\nabla u\|^2 \notag \\
 & \leq &(1+Cr^{\alpha/2})\int_{\partial B(0,r)} \!\!
\|\sqrt{A}\nabla u\|^2 
 =  (1+Cr^{\alpha/2})\int_{\partial B(0,r)}\!\!
 \|\nabla u\|_A^2 .\notag 
  \end{eqnarray}
Therefore, since $E'(r)=\int_{\partial B(0,r)}\|\nabla u\|_A^2$ we have proved for $r$ small enough,
 $$E(r)\leq (r+Cr^{1+\alpha/2})E'(r),$$
and we conclude with Lemma \ref{gronwall} below, applied with the exponent $\alpha/2 \in (0,1)$.
\end{proof}
%%%%%%%%%%%%%%%%%%%%%%%%%%%
%%%%%%%%%%%%%%%%%%%%%%%%%%%%

\begin{lemma}[Gronwall type, version 1] \label{gronwall} Assume that $E(r)$  admits a derivative a.e. on $[0,r_0]$, is absolutely continuous,  and satisfies the following inequality for some $\alpha \in (0,1)$
\begin{equation}
 E(r)\leq (r+Cr^{1+\alpha})E'(r), \quad \forall r \in[0,r_0]. \label{ineqdiff}
\end{equation}
Then the function
$$r \mapsto \frac{E(r)}{r}(1+Cr^{\alpha})^{\frac{1}{\alpha}}$$
is nondecreasing. As a consequence, the limit $\lim_{r\to 0}(E(r)/r)$ exists and is finite.  

\end{lemma}
\begin{proof} Observe that a primitive of $1/(r+Cr^{1+\alpha})$ is
\begin{equation}
\int \frac{1}{r+Cr^{1+\alpha}} dr\ =\ \ln\left(\frac{r}{(Cr^\alpha+1)^{\frac{1}{\alpha}}}\right)
\ =:\ h(r) . \label{primitive}
\end{equation}
Hence \eqref{ineqdiff} yields that
\[
\left(E(r)e^{-h(r)}\right)'\,=\, (-h'(r) E(r)+E'(r))e^{-h(r)}\,\ge\,0\,,
\]
% Then since $E(r)$ is nonnegative, the equation satisfied by $E(r)$ says that
% $$\frac{E'(r)}{E(r)}\geq  \frac{1}{r+Cr^{1+\alpha}},$$
% which implies, using \eqref{primitive},
% $$\ln(E(r))' -   \ln\left(\frac{r}{(Cr^\alpha+1)^{\frac{1}{\alpha}}}\right)' \geq 0$$
in other words,
$$r\mapsto \frac{E(r)}{r}(1+Cr^{\alpha})^{\frac{1}{\alpha}}$$
is nondecreasing.
Therefore the limit of $E(r)(1+Cr^{\alpha})^{\frac{1}{\alpha}}/r$
exists when $r$ goes to zero, and since 
$(1+Cr^{\alpha})^{\frac{1}{\alpha}}$ converges to $1$, we obtain the existence of limit also for $E(r)/r$. Now by monotonicity, this limit is necessarily finite since less than $\frac{E(r_0)}{r_0}(1+Cr_0^{\alpha})^{\frac{1}{\alpha}}$ which is finite for some $r_0$ fixed.
\end{proof}

We also have an analogue of Proposition \ref{prop11} in the context of general coefficients which is the proposition below.
\begin{proposition}[Energy estimate for general coefficients and second member]\label{prop11A} Let $u$ be a solution for the problem \eqref{problem1} with $\alpha$-H\"olderian coefficients $A$, $\lambda >0$ and $f,g \in L^\infty$, and assume that $K$ is a closed and connected set of finite length. For any point $x_0 \in K$ we denote
 $$E(r):=\int_{B_A(x_0,r)\setminus K}\|\nabla u \|^2_A dx.$$
We assume in addition that $K$ is locally-chord-arc at point $x_0$.  Then the limit $\lim_{r\to 0}(E(r)/r)$ exists and is finite.  
\end{proposition}

\begin{proof} We follow the proof of Proposition \ref{prop11}, with the changes already used in the proof of Proposition \ref{prop1bisA}. The main difference with the preceding propositions is that we arrive now to the inequality 
$$E(r)\leq  (r+Cr^{1+\alpha/2})E'(r) + CN(r)r^2.$$
then we conclude with the Lemma \ref{gronwall2} below.
\end{proof}
%%%%%%%%%%%%%%%%%%%%%%%
%%%%%%%%%%%%%%%%%%%%%%%

\begin{lemma}[Gronwall type, version 2]\label{gronwall2} Assume that $E(r)$  admits a derivative a.e. on $[0,r_0]$, is absolutely continuous, and satisfies the following inequality for some $\alpha \in (0,1)$
\begin{equation}
 E(r)\leq (r+Cr^{1+\alpha})E'(r)+CN(r)r^2, \quad \forall r \in[0,r_0], \label{lastineqdiff}
\end{equation}
with $N$ integrable on $(0,r_0)$. Then the limit $\lim_{r\to 0} E(r)/r$
   exists and is finite.
\end{lemma}
\begin{proof} Let us first find a particular solution of the inhomogeneous equation
\begin{equation}
G(r)=(r+Cr^{1+\alpha})G'(r)+CN(r)r^2. \label{equation1}
\end{equation}
For this purpose, recall (see Lemma \ref{gronwall}) that the solutions of the homogeneous first order linear equation 
$$f'(r)=\frac{1}{r+Cr^{1+\alpha}}\;f(r)$$
are given by 
$$f(r)=\lambda \frac{r}{(Cr^\alpha+1)^{1/\alpha}} \;, \quad \lambda \in \R.$$
Then from the method of ``variation of the constant" we deduce that a particular solution for equation \eqref{equation1} is 
$$G(r)=\lambda(r)\frac{r}{(Cr^\alpha+1)^{1/\alpha}}\;,$$
with $\lambda(r)= -C \int_0^r  N(t)(Ct^\alpha+1)^{\frac{1-\alpha}{\alpha}}\,dt$ (notice that $N(t)(Ct^\alpha+1)^{\frac{1-\alpha}{\alpha}}$ is integrable because $N$ is). In particular we have 
\begin{equation}
\lim_{r\to 0}\frac{G(r)}{r}=0. \label{estimationgood2}
\end{equation}
Now let us return to $E(r)$, which is assumed to satisfy \eqref{lastineqdiff}. If we subtract $G(r)$ in the equation \eqref{lastineqdiff} we get 
$$H(r)\leq (r+Cr^{1+\alpha})H'(r) \; ,$$
where $H(r)=E(r)-G(r)$. Therefore we can apply Lemma \ref{gronwall} to $H$ which gives the existence of the limit 
$$\lim_{r \to 0}\left(\frac{H(r)}{r}\right) <+\infty,$$
and we conclude using \eqref{estimationgood2}. 
\end{proof}

\section{Blow up}
\label{blowup}
Here we prove the second part of Theorem \ref{the} concerning the blow up sequence. Before going on with blow up limits at the origin, we start with a rigorous definition of $u(0)$. Indeed,  let $u$ be a solution for the problem \eqref{problem1} with $g \in L^\infty$, $(\lambda=f=0)$ or $(f \in L^{\infty}$ and $\lambda>0)$.  We suppose that $K$ is a closed and connected set satisfying the density condition \eqref{density_condition} at $0$. Let $R_r$ the family of rotations given by remark \ref{rotations} so that $r^{-1}R_r(K\cap B(0,r))$ converges to the segment $[-1,0]\times \{0\}$ when $r$ goes to $0$. For any $r$ small enough we define $A_r:= R_r^{-1}(B((r/2,0),r/4))$ and 
$$m_r:=\frac{1}{|A_r|}\int_{A_r}u(x) dx.$$

\begin{lemma}[Definition of $u(0)$]\label{defu0} The sequence  $m_r$ converges to some finite number that we will denote by $u(0)$. 
\end{lemma}

\begin{proof} We begin with a discrete sequence $r_n:=2^{-n}r_0$ for some $r_0$ small, $n \in \N$. In particular we assume $r_0$ small enough to have
\begin{equation}
\frac{1}{r_n}\int_{B(0,r_n)}\|\nabla u\|^2 dx \leq C \quad \forall n\in \N, \label{r0petit}
\end{equation}
for some constant $C$ that surely exists thank to Section \ref{Bonnet}.
 Since $r^{-1}R_r(K\cap B(0,r))$ converges to the segment $[-1,0]\times \{0\}$, we are sure that for $r_0$ small enough and for every $n$, the ball $B_n:=R_{r_n}^{-1}(B(r_n/2,0),3r_n/8)$ does not meet $K$ and contains both $A_{r_n}$ and $A_{r_{n+1}}$. We denote by $m_n$ the average of $u$ on $B_n$. Applying Poincar\'e inequality in $B_n$ yields
$$
|m_{r_n}-m_n| =   \left| \frac{1}{|A_{r_n}|}\int_{A_{r_n}} (u -m_n)dx \right| \leq  \frac{1}{|A_{r_n}|}  \int_{B_n}|u-m_n| \leq C \frac{1}{r_n}\int_{B_n}\|\nabla u\| dx
$$
and the same for $m_{r_{n+1}}$ so that at the end
\begin{equation}
|m_{r_n}-m_{r_{n+1}}|\leq C   \frac{1}{r_n}\int_{B_n}\|\nabla u\| dx
\leq C \left(\int_{B_n}\|\nabla u\|^2 dx\right)^{\frac{1}{2}} \leq C r_n^{1/2}  \label{cauchy}
\end{equation}
because of \eqref{r0petit}. In particular this implies that $m_{r_n}$ is a Cauchy sequence, thus converges to some limit $\ell \in\R$.
%which is surely finite, because $m_{n_0}$ is finite and passing to the limit in 
%\eqref{cauchy} we get $|m_{n_0}-\ell |\leq Cr_{n_0}^{1/2}$.
Now if $r_k$ is any other sequence converging to zero, we claim that the limit of $m_{r_k}$ is still equal to $\ell$. To see this it suffice to find a subsequence $r_{n_k}$ of $r_n$ such that $r_{n_k}/2 \leq r_k \leq r_{n_k}$ and compare $m_{r_k}$ with $m_{r_{n_k}}$ by the same way as we obtained \eqref{cauchy} and conclude that they must have same limit. 
 \end{proof}

\begin{remark}\label{moyennesbis}  In the future it will be convenient to introduce another type of averages on circles, namely
$$\tilde{m}_r:=\frac{1}{\tilde{A}_r}\int_{\tilde{A}_r}u \;d\Hh,$$
with 
$$\tilde{A}_r:= B((r,0),\frac{r}{4})\cap \partial B(0,r)$$
It is easily checked that the sequence of $\tilde{m}_r$ are aslo converging to $u(0)$, i.e. has same limit as $m_r$.
\end{remark}

We are now ready to prove the last part of Theorem \ref{the}. 
\begin{theorem}[Convergence of the blow-up sequence]\label{prop2} Let $u$ be a solution for the problem \eqref{problem1} with $g \in L^\infty$, $(\lambda=f=0)$ or $(f \in L^{\infty}$ and $\lambda>0)$.  We suppose that $K$ is a closed and connected set satisfying the density condition \eqref{density_condition} at the origin. We denote by $u(0)$ the real number given by Lemma  \eqref{defu0}. Let $R_r$ the family of rotations given by Remark \ref{rotations} so that $r^{-1}R_r(K\cap B(0,r))$ converges to some segment $\Sigma_0$ when $r$ goes to $0$. If $(r,\theta)$ are the polar coordinates such that $(\sqrt{A(0)})^{-1}(\Sigma_0)=(\R^-\times\{\pi\})$ we denote by $v_0$ the function defined in polar coordinates by
$$v_0(r,\theta) := \sqrt{\frac{2C_0r}{\pi}} \sin(\theta/2).$$
Then 
$$u_r:=r^{-\frac{1}{2}}(u(rR_r^{-1}x) -u(0)) \underset{r\to 0}{\longrightarrow}  v_0\circ \sqrt{A(0)},$$
where the constant $C_0$ is given by
$$C_0=\lim_{r\to 0}\left(\frac{1}{\det(\sqrt{A(0)})r}\int_{B_A(0,r)\setminus K} \|\nabla u\|_A^2 dx\right),$$
and the convergence holds strongly in $L^2(B(0,1))$ for both $u_r$ and $\nabla u_r$.
\end{theorem}

\begin{proof} We know that $K_r:=\frac{1}{r}R_r(K)$ converges to the half-line $\mathbb{R}^- \times \{0\}$ locally in $\R^2$ for the Hausdorff distance. To simplify the notations and  without loss of generality, in the sequel we will identify  $u$ with $u\circ R_r^{-1}$, and $K$ with $R_r(K)$ so that we can assume that $R_r=Id$ for all $r$. We can also assume that $u(0)=0$ and as before, it is enough to consider the case when $A(0)=Id$ because the general case follows using the change of variable of Proposition \ref{chang}. 

%%%%%%%%%%%%%%%%%%%%%%%%%%%%%%%%%%%%%%%%%%%%%%%
As in the proof of Lemma \ref{defu0}, for any $r$ we denote by $m_r$ the average of $u$ on the ball $B((r/2,0),r/4)$. Then we consider the function $u_r(x):=r^{-\frac{1}{2}}(u(rx)-m_r)$ defined in $\frac{1}{r}(\Omega\backslash K)$. The domain $\frac{1}{r}(\Omega\backslash K)$ converges to $\R^2\setminus K_0$ with $K_0:=\R^{-}\times \{0\}$.

We will prove that $u_r$ converges, in some sense that will be given later, to function in $\R^2 \setminus K_0$ that satisfies a certain Neumann problem. In the sequel we will work up to subsequences, but this will not be restrictive in the end by uniqueness of the limit.

The starting point is  that $\nabla u_r$ is uniformly bounded in $L^2(B(0,2))$ (we start working in $B(0,2)$  for security but the real interesting ball will be $B(0,1)$). Indeed, $\nabla u_r(x)=\sqrt{r}\nabla u(rx)$, 
$$\int_{B(0,2)\setminus K_r}\|\nabla u_r\|^2 dx=\int_{B(0,2)\setminus K_r}r\|\nabla u(rx)\|^2dx = \frac{1}{r}\int_{B(0,2r)\setminus K}\|\nabla u(x)\|^2dx.$$
{}From Proposition \ref{prop1}, we know that $\frac{1}{r}\int_{B(0,r)}\|\nabla u(x) \|_A^2dx$ converges to $C_0$ and we deduce (using the coerciveness of $A$), that $\nabla u_r$ is uniformly bounded in $L^2(B(0,2))$.

Therefore we can extract a subsequence such that $\nabla u_r$ converges to some $h$, weakly in $L^2(B(0,2))$, and
\begin{equation}\label{semicontinue}
\int_{B(0,1)}\|h \|^2 \leq \liminf_{r\to 0} \int_{B(0,1)}\|\nabla u_r\|^2 dx \leq C.
\end{equation}
Next we want to prove that in compact sets of $B(0,2)\setminus K_0$, the convergence is much better. For this purpose we introduce for any $a>0$ 
$$U(a):=\{x \in B(0,2); d(x,K_0)>a  \}.$$
The sequence  $u_r$ is uniformly bounded in $H^1(U(a))$ for any $a$. Therefore taking a sequence $a_n\to 0$, extracting some subsequence of $u_r$ and using a diagonal argument we can find a subsequence of $u_r$, not relabeled, that converges weakly in $H^1$ and strongly in $L^2$ in any of the domains $U(a)$. In other words, this subsequence $u_r$ converges weakly in $H^1_{loc}(B(0,2)\setminus K_0)$ and strongly in $L^2_{loc}(B(0,2),\setminus K_0)$ to some function $u_0 \in H^1_{loc}(B(0,2)\setminus K_0)$.  By uniqueness of the limit we must have that $\nabla u_0=h$ a.e. in $B(0,2)$ and therefore \eqref{semicontinue} reads
\begin{equation}
\int_{B(0,1)}\|\nabla u_0\|^2 \leq \liminf_{r\to 0} \int_{B(0,1)}\|\nabla u_r\|^2 dx \leq C.  \label{semicontinue2}
\end{equation}

Now we want to prove that $u_0$ is a minimizer for the Dirichlet energy, and at the same time prove that the convergence hold strongly in $L^2(B(0,1))$ both for $u_r$ and $\nabla u_r$. To do this we consider any function $v\in H^1(B(0,1)\setminus K_0)$ with
$v\equiv u_0$ in $B(0,1)\setminus B(0,1-\delta)$ and $v\equiv 0$ in $B(0,\eta)$, for some small $\delta>0$. The family of all such functions $v$ is dense in the space of functions of $H^1(B(0,1)\setminus K_0)$ with trace equal to $u_0$ on $\partial B(0,1)\setminus K_0$ and therefore to prove that $u_0$ is a minimizer, it is enough to prove that 
$$\int_{B(0,1)}\|\nabla u_0\|^2 dx \leq \int_{B(0,1)}\|\nabla v\|^2 dx$$
for all such functions $v$. 

We denote by $N_r(s)$ the number of points of $K_r\cap \partial B(0,s)$. As already used before, since by assumption $\Hh(K_r\cap B(0,1))$ converges to $1$ and 
$$1\leq \int_{0}^1N_r(s)ds \leq \Hh(K_r\cap B(0,1)),$$
we can extract a subsequence  such that  $N_{r}(s) \to 1$ a.e.
Then Fatou's lemma yields
\begin{equation}
\int_0^1 \liminf_r \int_{\partial B_s} \|\nabla u_{r}\|^2 d\Hh\,ds
\ \le\
\liminf_r \int_0^1 \int_{\partial B_s} \|\nabla u_{r}\|^2 d\Hh\,ds
= C_1, %% la limite des \int |\nabla u_\rho|^2 
\label{fatou}
\end{equation}
where $C_1$ is closely related to $C_0$.  This will allows us later to find a good radius $s$ for which both $N(s)=1$ and $\int_{\partial B_s}\|\nabla u_r\|^2 d\Hh$ is uniformly bounded.

At this stage we only know that $\nabla u_r$ converges weakly in $L^2$ to $\nabla u_0$. On the other hand, up to a further subsequence, we can find a measure $\mu$ such that $ |\nabla u_r|^2 dx $ weakly-$\star$ converges  to $\mu$.
Let $x\in B(0,2)$, $\rho>0$ such that $\overline{B(x,\rho)}\subset B(0,2)\setminus K_0$.
Let $\psi$ be a smooth cutoff, with support in $B(x,\rho)$, and equal to $1$ in $B(x,\rho/2)$.
Then we can write that
\begin{equation}\label{weakbrho}
\int_{B(x,\rho)} (A_r\nabla u_r)\cdot\nabla (\psi (u_r-u_0)) + r^2\lambda u_r(u_r-u_0)\psi - r^{3/2} f_r (u_r-u_0)\psi
\ =\ 0
\end{equation}
where $A_r(x)=A(rx)$, $f_r(x)=f(rx)-\lambda m_r$, and (taking the limit in the ``first''
$u_r$ while freezing the test function $(u_r-u_0)\psi$, and using the weak convergence
in $H^1(B(x,\rho)$ of $u_r$ to $u_0$):
\begin{equation}\label{weakbrho0}
\int_{B(x,\rho)} (\nabla u_0)\cdot\nabla (\psi (u_r-u_0)) \ =\ 0\,.
\end{equation}
Taking the difference of \eqref{weakbrho} and \eqref{weakbrho0},
and using the fact that $u_r\to u_0$ strongly in $L^2(B_r)$,
$\nabla u_r$ is uniformly bounded in $L^2(B_r)^2$, and $A_r\to Id$ uniformly, we obtain that
\[
\lim_{r\to 0} \int_{B(x,\rho/2)} \|\nabla u_r-\nabla u_0\|^2\,dx\ =\ 0
\]
so that clearly, $\mu\restr (B(0,2)\setminus K_0) = \|\nabla u_0\|^2\,dx$: if $\mu$ has
a singular part it must be concentrated on $K_0$.
 Moreover, we have $\mu(\{(-s,0)\})=0$
for all $s\in [0,2)$ but a countable number. (Observe that using any other test function
in \eqref{weakbrho} and passing to the limit, we easily deduce that $u_0$ is harmonic
in $B(0,2)\setminus K_0$, but this will also be a consequence of the minimality of the
Dirichlet energy which will soon be shown).

Now from \eqref{fatou} we may choose $s$,  $1-\delta<s<1$, so that
$\mu (\{-s,0\}) = 0$,
$N_{r}(s)=1$ for all $r$ large enough,  and
\[
\liminf_r \int_{\partial B_s} \|\nabla u_{r}\|^2 d\Hh \ <\ +\infty
\]
In particular, upon extracting a further subsequence, we may assume that
\[
\sup_r \int_{\partial B_s} \|\nabla u_{r}\|^2 d\Hh \ <\ +\infty.
\]
Then, by Sobolev's embedding, and using the fact that the averages
$\tilde{m}_r$ are uniformly bounded (see Remark \eqref{moyennesbis}), we deduce that there exists 
$C>0$ such that
\[
\|u_{r}\|_{L^\infty(\partial B_s)} \ \le\ C.
\]
We now consider any constant $M>C$  and define 
\[
u^M_{r} \ =\ (-M \vee (u_{r} \wedge M))
\]
we  have that $u^M_{r}\to u_0^M$ in
        $L^2_{loc}(B(0,1)\setminus K_0)$,  where $u_0^M$ is naturally defined as being $u_0^M:=(-M \vee (u_{0} \wedge M))$. Up to a subsequence the convergence holds almost everywhere. 
        But now, since the functions are uniformly bounded, it converges also strongly in
$L^2(B(0,1)\setminus K_0)$.  %%This  comes from the dominated convergence theorem.

Now, from the original function $v \in H^1(B(0,1)\setminus K_0)$, we want to construct a function $v_r \in H^1(B(0,1)\setminus K_r)$ not much different from $v$. We denote by $C^\pm_r$ the connected components of $(B(0,1)\setminus K_r)\cap \{ x\leq 0\}$ containing $(-1/2,\pm 1/2)$ and we define $v_r(x,y)$ as follows. In $B(0,1)\cap \{x>0\}$ we set $v_r(x,y)=v(x,y)$.
$$
\begin{array}{cc}
\text{ In } C^+_r, & v_r(x,y)=
\left\{
\begin{array}{cc}
v(x,y) & \text{ if } y\geq 0 \\
v(x,-y) & \text{ otherwise. }
\end{array}
\right. \\
\text{ In } C^-_r, & v_r(x,y)=
\left\{
\begin{array}{cc}
v(x,y) & \text{ if } y\leq 0 \\
v(x,-y) & \text{ otherwise. }
\end{array}
\right. \\
\end{array}
$$
And finally $v_r=0$ everywhere else (i.e. in $B(0,1)\cap \{x\leq 0\} \setminus (C_n^+\cap C_n^-)$). Then it is easy to see that $v_r \in H^1(B(0,1)\setminus K_r)$, converges strongly to $v$ in $L^2$ and ${\bf 1}_{B(0,1)\setminus K_r}\nabla v_r$ converges strongly to ${\bf 1}_{B(0,1)\setminus K_0}\nabla v$ in $L^2(B(0,1))$. However, by this procedure the trace on $\partial B(0,1)$ is not necessarily preserved.  

To get rid of that  we let $\varepsilon<s-(1-\delta)$, we pick a smooth cut-off $\psi_\e$ with compact support in $B_s$, 
$0\le \psi_\e\le 1$ and $\psi_\e\equiv 1$ in $B_{s-\e}$, and
we let
\[
v_r^\e \ =\ \psi_\e v_r + (1-\psi_\e) u^M_{r}
\]
which converges strongly in $L^2$ to $\psi_\e v + (1-\psi_\e) u_0^M$ as $r\to 0$. 
Next we  write, since $v_r^\e = u^M_{r} = u_{r}$ on $\partial B_s$ and $u_r$ is a minimizer,
%%% ICI A_k = A(\rho_k y) ; f_k = f(\rho_k y)
\[
\int_{B_s} (A_r\nabla u_{r})\cdot \nabla u_{r}
+ \lambda r^2 u_{r}^2 - 2r^{3/2} f_r u_{r}\,dx
\le
  \int_{B_s} (A_r\nabla v_r^\e)\cdot \nabla v_r^\e
+ \lambda r^2 (v_r^\e)^2 - 2r^{3/2} f_r v_r^\e\,dx\,.
\]

Recall that $|f_r|\leq C$ and  $|u_{r}|\le C/\sqrt{r}$ (by definition) so that
$2r^{3/2} f_r u_{r}= o(r)$, and we also easily check that
\[
\int_{B_s} \lambda r^2 (v_r^\e)^2 - 2r^{3/2} f_r v_r^\e\,dx 
\ =\ o (r)
\]
hence we focus on the other terms: we write for $\delta>0$ small,
\begin{multline*}
\int_{B_s} (A_r\nabla u_{r})\cdot \nabla u_{r} \,dx 
\\ \le\ 
(1+\eta)\int_{B_s} \psi_\e^2(A_r\nabla v_r)\cdot \nabla v_r\,dx
\ + C'/\eta \int_{B_s}\|\nabla\psi_\e\|^2|v_r-u^M_{r}|^2\,dx
\\ + \ C'/\eta \int_{B_s} (1-\psi_\e)^2\|\nabla u_{r}\|^2\,dx
+ o(r)
\end{multline*}

Then sending $r\to 0$ we obtain
\begin{multline*}
\int_{B_s} \|\nabla u_0\|^2\ \le\
(1+\eta)\int_{B_s} \psi_\e^2\|\nabla v\|^2 \,dx
\\
+\ C'/\eta \int_{B_s}\|\nabla\psi_\e\|^2|v-u^M_0|^2\,dx
+ C'/\eta \mu(B_\e(-s,0))
\end{multline*}
but on the support of $\nabla \psi$, $v-u^M_0$ is equal to $(1-\psi_\e)(u_0^M-u_0)$. Therefore letting $M$ tend to $+\infty$ we get 
\[
\int_{B_s} \|\nabla u_0\|^2\ \le\
(1+\eta)\int_{B_s} \psi_\e^2\|\nabla v\|^2 \,dx + 
C'/\eta \mu(B_\e(-s,0))
\]
finaly letting $\e\to 0$, then  $\eta \to 0$, and adding the integral over $B(0,1)\setminus B(0,s)$ on both sides (where $v$ and $u_0$ actually coincide) we get the desired inequality, namely
\[
\int_{B(0,1)} \|\nabla u_0\|^2 dx\ \le\
\int_{B(0,1)} \|\nabla v\|^2 \,dx 
\]
which proves that $u_0$ is a minimizer. Moreover taking the particular choice $v=u_0$ in the same argument as before would give 
$$\limsup_{r\to 0} \int_{B(0,1)}\|\nabla u_r\|^2dx \ \le \  \int_{B(0,1)} \|\nabla u_0\|^2dx$$
and this toghether with \eqref{semicontinue2}, implies the convergence of norms, which by the weak convergence yields the strong convergence in $L^2$ for the gradients, as desired.

Finally all that we did in $B(0,1)$ could be done in any $B(0,R)$ for $R$ as large as we want, which gives a definition of $u_0$ in $\R^2\setminus K_0$. Moreover $u_0$ is of constant normalized energy. In other words we claim that 
$s\mapsto\frac{1}{s}\int_{B(0,s)}\|\nabla u_0\|^2$
is constant in $s$, identically equal to $C_0$. Indeed, by the strong convergence in $L^2$ of $\nabla u_r$,  the value of $\frac{1}{s}\int_{B(0,s)}\|\nabla u_0\|^2$ is given by 
$$\lim_{r\to 0} \frac{1}{s}\int_{B(0,s)}\|\nabla u_r\|^2,$$
which we actually claim to be equal to $C_0$:  a change of variable gives
\begin{eqnarray}
\int_{B(0,s)}\|\nabla u_r\|^2 &= & \frac{1}{r}\int_{B(0,rs)}\|\nabla u\|^2\notag \\
&=&  \frac{1}{r}\int_{B(0,rs)}\|\nabla u \|_A^2 + \frac{1}{r} \int_{B(0,rs)}\langle (Id-A)\nabla u , \nabla u \rangle. \label{eq12}
\end{eqnarray}
The first term in \eqref{eq12} converges to $sC_0$ and the second term converges to zero because less than   $\|Id-A\|_{L^\infty(B(0,r))}$ times something bounded.

%%%%%%%%%%%%%%%%%%%%%%%%%%%%%%%%%%%%%%%%%%%%%%%%%%%%%%%%%%%%%%%%
The latter implies that $u_0$ is the cracktip function. More precisely, we claim now that
\begin{equation}
u_0=\sqrt{\frac{2C_0r}{\pi}}\sin(\theta/2). \label{crak}
\end{equation}
We shall give two different arguments for \eqref{crak}. The first one is very nice and due to Bonnet: returning to the proof of the monotonicity Lemma applied to $u_0$, which says that $s\mapsto\frac{1}{s}\int_{B(0,s)}\|\nabla u_0\|^2$ must be increasing (Proposition \ref{prop1}), since $s\mapsto\frac{1}{s}\int_{B(0,s)}\|\nabla u_0\|^2$ is actually constant in $s$, all the inequalities in the proof are equalities. In particular $u_0$ must be the optimal function in Wirtinger inequality, thus it is the famous $C\sqrt{r}\sin(\theta/2)$ function.

The second argument is to decompose $u_0$ in spherical harmonics, i.e. as a sum of homogeneous harmonic functions in the complement of the half line $K_0$, which Neumann boundary conditions on $K_0$. Now using that  $s\mapsto\frac{1}{s}\int_{B(0,s)}\|\nabla u_0\|^2$ is constant we can kill all the terms of degree different from $1/2$ by taking blow-up and blow-in limits. This implies that $u_0$ must be homogeneous of degree $1/2$, and from this information it is not difficult to deduce \eqref{crak} by looking at $u_0$ on the unit circle.

 Then, the exact constant $C:=\sqrt{\frac{2C_0}{\pi}}$ in front of the sinus can be easily computed by hand with the formulas
$$\frac{\partial u_0}{\partial \tau}=\frac{1}{r}\frac{\partial u_0}{\partial \theta}=C\frac{1}{2\sqrt{r}}\cos(\theta/2)\quad \quad \text{ and }\quad \quad \frac{\partial u_0}{\partial r} = C\frac{1}{2\sqrt{r}}\sin(\theta/2).$$
It comes
$$RC_0=\int_{B(0,R)}\|\nabla u_0\|^2=\int_{B(0,R)}\Big|\frac{\partial u_0}{\partial \tau}\Big|^2+\Big|\frac{\partial u_0}{\partial r}\Big|^2=\int_{0}^{R}\int_{-\pi}^\pi \frac{C^2}{4}drd\theta=C^2R\frac{\pi}{2}$$
thus $C=\sqrt{\frac{2C_0}{\pi}}$.

Finally, originally $u_r$ was converging to $\sqrt{\frac{2C_0}{\pi}}\sin(\theta/2)$ up to subsequences, but by uniqueness of the limit we conclude that the whole sequence converges to this function and this achieves the proof.
\end{proof}

%For acknowledgements section, please don't number the section, please begin it with \section*{Acknowledgements}
%\section*{Acknowledgments} We would like to thank you for \textbf{following
%the instructions above} very closely in advance. It will definitely
%save us lot of time and expedite the process of your paper's
%publication.

% You may incorporate your references as follows in your main tex file.
% Using BibTex is not recommended but can be handled.
\bibliographystyle{plain}
\bibliography{biblio}

\end{document}

%% file: dessin1.tex
\ifx\JPicScale\undefined\def\JPicScale{1}\fi
\unitlength \JPicScale mm
\begin{picture}(96.96,63)(0,0)
\put(96.96,30.11){\makebox(0,0)[cc]{}}

\put(94,31){\makebox(0,0)[cc]{$K$ looks like a diameter}}

\put(89,9){\makebox(0,0)[cc]{$K$ looks like a radius}}

\linethickness{0.3mm}
\put(19,24){\line(1,0){26}}
\linethickness{0.3mm}
\put(25,24){\line(0,1){4}}
\linethickness{0.3mm}
\put(25,28){\line(1,0){40}}
\linethickness{0.3mm}
\multiput(64.5,24.24)(0.01,-0.49){1}{\line(0,-1){0.49}}
\multiput(64.5,23.27)(0.01,0.49){1}{\line(0,1){0.49}}
\multiput(64.48,22.78)(0.02,0.49){1}{\line(0,1){0.49}}
\multiput(64.44,22.29)(0.03,0.49){1}{\line(0,1){0.49}}
\multiput(64.4,21.8)(0.05,0.49){1}{\line(0,1){0.49}}
\multiput(64.34,21.32)(0.06,0.49){1}{\line(0,1){0.49}}
\multiput(64.27,20.83)(0.07,0.48){1}{\line(0,1){0.48}}
\multiput(64.19,20.35)(0.08,0.48){1}{\line(0,1){0.48}}
\multiput(64.09,19.87)(0.1,0.48){1}{\line(0,1){0.48}}
\multiput(63.98,19.4)(0.11,0.48){1}{\line(0,1){0.48}}
\multiput(63.86,18.93)(0.12,0.47){1}{\line(0,1){0.47}}
\multiput(63.72,18.46)(0.13,0.47){1}{\line(0,1){0.47}}
\multiput(63.58,17.99)(0.15,0.47){1}{\line(0,1){0.47}}
\multiput(63.42,17.53)(0.16,0.46){1}{\line(0,1){0.46}}
\multiput(63.25,17.07)(0.17,0.46){1}{\line(0,1){0.46}}
\multiput(63.06,16.62)(0.09,0.23){2}{\line(0,1){0.23}}
\multiput(62.87,16.17)(0.1,0.22){2}{\line(0,1){0.22}}
\multiput(62.66,15.72)(0.1,0.22){2}{\line(0,1){0.22}}
\multiput(62.44,15.29)(0.11,0.22){2}{\line(0,1){0.22}}
\multiput(62.21,14.86)(0.12,0.22){2}{\line(0,1){0.22}}
\multiput(61.97,14.43)(0.12,0.21){2}{\line(0,1){0.21}}
\multiput(61.72,14.01)(0.13,0.21){2}{\line(0,1){0.21}}
\multiput(61.45,13.6)(0.13,0.21){2}{\line(0,1){0.21}}
\multiput(61.18,13.19)(0.14,0.2){2}{\line(0,1){0.2}}
\multiput(60.89,12.79)(0.14,0.2){2}{\line(0,1){0.2}}
\multiput(60.6,12.4)(0.15,0.2){2}{\line(0,1){0.2}}
\multiput(60.29,12.02)(0.1,0.13){3}{\line(0,1){0.13}}
\multiput(59.97,11.64)(0.11,0.13){3}{\line(0,1){0.13}}
\multiput(59.65,11.27)(0.11,0.12){3}{\line(0,1){0.12}}
\multiput(59.31,10.92)(0.11,0.12){3}{\line(0,1){0.12}}
\multiput(58.97,10.56)(0.12,0.12){3}{\line(0,1){0.12}}
\multiput(58.61,10.22)(0.12,0.11){3}{\line(1,0){0.12}}
\multiput(58.25,9.89)(0.12,0.11){3}{\line(1,0){0.12}}
\multiput(57.87,9.57)(0.12,0.11){3}{\line(1,0){0.12}}
\multiput(57.49,9.25)(0.13,0.1){3}{\line(1,0){0.13}}
\multiput(57.1,8.95)(0.13,0.1){3}{\line(1,0){0.13}}
\multiput(56.7,8.66)(0.2,0.15){2}{\line(1,0){0.2}}
\multiput(56.3,8.37)(0.2,0.14){2}{\line(1,0){0.2}}
\multiput(55.89,8.1)(0.21,0.14){2}{\line(1,0){0.21}}
\multiput(55.47,7.84)(0.21,0.13){2}{\line(1,0){0.21}}
\multiput(55.04,7.59)(0.21,0.13){2}{\line(1,0){0.21}}
\multiput(54.6,7.35)(0.22,0.12){2}{\line(1,0){0.22}}
\multiput(54.16,7.12)(0.22,0.11){2}{\line(1,0){0.22}}
\multiput(53.72,6.9)(0.22,0.11){2}{\line(1,0){0.22}}
\multiput(53.27,6.69)(0.23,0.1){2}{\line(1,0){0.23}}
\multiput(52.81,6.49)(0.23,0.1){2}{\line(1,0){0.23}}
\multiput(52.35,6.31)(0.23,0.09){2}{\line(1,0){0.23}}
\multiput(51.88,6.14)(0.47,0.17){1}{\line(1,0){0.47}}
\multiput(51.41,5.98)(0.47,0.16){1}{\line(1,0){0.47}}
\multiput(50.93,5.83)(0.48,0.15){1}{\line(1,0){0.48}}
\multiput(50.45,5.69)(0.48,0.14){1}{\line(1,0){0.48}}
\multiput(49.97,5.57)(0.48,0.12){1}{\line(1,0){0.48}}
\multiput(49.48,5.45)(0.49,0.11){1}{\line(1,0){0.49}}
\multiput(48.99,5.35)(0.49,0.1){1}{\line(1,0){0.49}}
\multiput(48.5,5.27)(0.49,0.09){1}{\line(1,0){0.49}}
\multiput(48,5.19)(0.5,0.08){1}{\line(1,0){0.5}}
\multiput(47.5,5.13)(0.5,0.06){1}{\line(1,0){0.5}}
\multiput(47,5.08)(0.5,0.05){1}{\line(1,0){0.5}}
\multiput(46.5,5.04)(0.5,0.04){1}{\line(1,0){0.5}}
\multiput(46,5.01)(0.5,0.03){1}{\line(1,0){0.5}}
\multiput(45.5,5)(0.5,0.01){1}{\line(1,0){0.5}}
\put(45,5){\line(1,0){0.5}}
\multiput(44.5,5.01)(0.5,-0.01){1}{\line(1,0){0.5}}
\multiput(43.99,5.04)(0.5,-0.03){1}{\line(1,0){0.5}}
\multiput(43.49,5.08)(0.5,-0.04){1}{\line(1,0){0.5}}
\multiput(42.99,5.13)(0.5,-0.05){1}{\line(1,0){0.5}}
\multiput(42.49,5.19)(0.5,-0.06){1}{\line(1,0){0.5}}
\multiput(42,5.27)(0.5,-0.08){1}{\line(1,0){0.5}}
\multiput(41.5,5.35)(0.5,-0.09){1}{\line(1,0){0.5}}
\multiput(41.01,5.45)(0.49,-0.1){1}{\line(1,0){0.49}}
\multiput(40.52,5.57)(0.49,-0.11){1}{\line(1,0){0.49}}
\multiput(40.03,5.69)(0.49,-0.12){1}{\line(1,0){0.49}}
\multiput(39.55,5.83)(0.48,-0.14){1}{\line(1,0){0.48}}
\multiput(39.07,5.98)(0.48,-0.15){1}{\line(1,0){0.48}}
\multiput(38.59,6.14)(0.48,-0.16){1}{\line(1,0){0.48}}
\multiput(38.12,6.31)(0.47,-0.17){1}{\line(1,0){0.47}}
\multiput(37.65,6.49)(0.23,-0.09){2}{\line(1,0){0.23}}
\multiput(37.19,6.69)(0.23,-0.1){2}{\line(1,0){0.23}}
\multiput(36.73,6.9)(0.23,-0.1){2}{\line(1,0){0.23}}
\multiput(36.28,7.12)(0.23,-0.11){2}{\line(1,0){0.23}}
\multiput(35.83,7.35)(0.22,-0.11){2}{\line(1,0){0.22}}
\multiput(35.39,7.59)(0.22,-0.12){2}{\line(1,0){0.22}}
\multiput(34.96,7.84)(0.22,-0.13){2}{\line(1,0){0.22}}
\multiput(34.53,8.1)(0.21,-0.13){2}{\line(1,0){0.21}}
\multiput(34.11,8.37)(0.21,-0.14){2}{\line(1,0){0.21}}
\multiput(33.7,8.66)(0.21,-0.14){2}{\line(1,0){0.21}}
\multiput(33.29,8.95)(0.2,-0.15){2}{\line(1,0){0.2}}
\multiput(32.9,9.25)(0.13,-0.1){3}{\line(1,0){0.13}}
\multiput(32.51,9.57)(0.13,-0.1){3}{\line(1,0){0.13}}
\multiput(32.13,9.89)(0.13,-0.11){3}{\line(1,0){0.13}}
\multiput(31.75,10.22)(0.12,-0.11){3}{\line(1,0){0.12}}
\multiput(31.39,10.56)(0.12,-0.11){3}{\line(1,0){0.12}}
\multiput(31.03,10.92)(0.12,-0.12){3}{\line(1,0){0.12}}
\multiput(30.69,11.27)(0.12,-0.12){3}{\line(0,-1){0.12}}
\multiput(30.35,11.64)(0.11,-0.12){3}{\line(0,-1){0.12}}
\multiput(30.02,12.02)(0.11,-0.13){3}{\line(0,-1){0.13}}
\multiput(29.71,12.4)(0.11,-0.13){3}{\line(0,-1){0.13}}
\multiput(29.4,12.79)(0.1,-0.13){3}{\line(0,-1){0.13}}
\multiput(29.11,13.19)(0.15,-0.2){2}{\line(0,-1){0.2}}
\multiput(28.82,13.6)(0.14,-0.2){2}{\line(0,-1){0.2}}
\multiput(28.54,14.01)(0.14,-0.21){2}{\line(0,-1){0.21}}
\multiput(28.28,14.43)(0.13,-0.21){2}{\line(0,-1){0.21}}
\multiput(28.03,14.86)(0.13,-0.21){2}{\line(0,-1){0.21}}
\multiput(27.79,15.29)(0.12,-0.22){2}{\line(0,-1){0.22}}
\multiput(27.56,15.72)(0.12,-0.22){2}{\line(0,-1){0.22}}
\multiput(27.34,16.17)(0.11,-0.22){2}{\line(0,-1){0.22}}
\multiput(27.13,16.62)(0.1,-0.22){2}{\line(0,-1){0.22}}
\multiput(26.93,17.07)(0.1,-0.23){2}{\line(0,-1){0.23}}
\multiput(26.75,17.53)(0.09,-0.23){2}{\line(0,-1){0.23}}
\multiput(26.58,17.99)(0.17,-0.46){1}{\line(0,-1){0.46}}
\multiput(26.42,18.46)(0.16,-0.47){1}{\line(0,-1){0.47}}
\multiput(26.28,18.93)(0.15,-0.47){1}{\line(0,-1){0.47}}
\multiput(26.14,19.4)(0.13,-0.47){1}{\line(0,-1){0.47}}
\multiput(26.02,19.87)(0.12,-0.48){1}{\line(0,-1){0.48}}
\multiput(25.91,20.35)(0.11,-0.48){1}{\line(0,-1){0.48}}
\multiput(25.81,20.83)(0.1,-0.48){1}{\line(0,-1){0.48}}
\multiput(25.73,21.32)(0.08,-0.48){1}{\line(0,-1){0.48}}
\multiput(25.66,21.8)(0.07,-0.49){1}{\line(0,-1){0.49}}
\multiput(25.6,22.29)(0.06,-0.49){1}{\line(0,-1){0.49}}
\multiput(25.56,22.78)(0.05,-0.49){1}{\line(0,-1){0.49}}
\multiput(25.52,23.27)(0.03,-0.49){1}{\line(0,-1){0.49}}
\multiput(25.5,23.76)(0.02,-0.49){1}{\line(0,-1){0.49}}
\multiput(25.5,24.24)(0.01,-0.49){1}{\line(0,-1){0.49}}
\multiput(25.5,24.24)(0.01,0.49){1}{\line(0,1){0.49}}
\multiput(25.5,24.73)(0.02,0.49){1}{\line(0,1){0.49}}
\multiput(25.52,25.22)(0.03,0.49){1}{\line(0,1){0.49}}
\multiput(25.56,25.71)(0.05,0.49){1}{\line(0,1){0.49}}
\multiput(25.6,26.2)(0.06,0.49){1}{\line(0,1){0.49}}
\multiput(25.66,26.68)(0.07,0.48){1}{\line(0,1){0.48}}
\multiput(25.73,27.17)(0.08,0.48){1}{\line(0,1){0.48}}
\multiput(25.81,27.65)(0.1,0.48){1}{\line(0,1){0.48}}
\multiput(25.91,28.13)(0.11,0.48){1}{\line(0,1){0.48}}
\multiput(26.02,28.6)(0.12,0.47){1}{\line(0,1){0.47}}
\multiput(26.14,29.07)(0.13,0.47){1}{\line(0,1){0.47}}
\multiput(26.28,29.54)(0.15,0.47){1}{\line(0,1){0.47}}
\multiput(26.42,30.01)(0.16,0.46){1}{\line(0,1){0.46}}
\multiput(26.58,30.47)(0.17,0.46){1}{\line(0,1){0.46}}
\multiput(26.75,30.93)(0.09,0.23){2}{\line(0,1){0.23}}
\multiput(26.94,31.38)(0.1,0.22){2}{\line(0,1){0.22}}
\multiput(27.13,31.83)(0.1,0.22){2}{\line(0,1){0.22}}
\multiput(27.34,32.28)(0.11,0.22){2}{\line(0,1){0.22}}
\multiput(27.56,32.71)(0.12,0.22){2}{\line(0,1){0.22}}
\multiput(27.79,33.14)(0.12,0.21){2}{\line(0,1){0.21}}
\multiput(28.03,33.57)(0.13,0.21){2}{\line(0,1){0.21}}
\multiput(28.28,33.99)(0.13,0.21){2}{\line(0,1){0.21}}
\multiput(28.55,34.4)(0.14,0.2){2}{\line(0,1){0.2}}
\multiput(28.82,34.81)(0.14,0.2){2}{\line(0,1){0.2}}
\multiput(29.11,35.21)(0.15,0.2){2}{\line(0,1){0.2}}
\multiput(29.4,35.6)(0.1,0.13){3}{\line(0,1){0.13}}
\multiput(29.71,35.98)(0.11,0.13){3}{\line(0,1){0.13}}
\multiput(30.03,36.36)(0.11,0.12){3}{\line(0,1){0.12}}
\multiput(30.35,36.73)(0.11,0.12){3}{\line(0,1){0.12}}
\multiput(30.69,37.08)(0.12,0.12){3}{\line(0,1){0.12}}
\multiput(31.03,37.44)(0.12,0.11){3}{\line(1,0){0.12}}
\multiput(31.39,37.78)(0.12,0.11){3}{\line(1,0){0.12}}
\multiput(31.75,38.11)(0.12,0.11){3}{\line(1,0){0.12}}
\multiput(32.13,38.43)(0.13,0.1){3}{\line(1,0){0.13}}
\multiput(32.51,38.75)(0.13,0.1){3}{\line(1,0){0.13}}
\multiput(32.9,39.05)(0.2,0.15){2}{\line(1,0){0.2}}
\multiput(33.3,39.34)(0.2,0.14){2}{\line(1,0){0.2}}
\multiput(33.7,39.63)(0.21,0.14){2}{\line(1,0){0.21}}
\multiput(34.11,39.9)(0.21,0.13){2}{\line(1,0){0.21}}
\multiput(34.53,40.16)(0.21,0.13){2}{\line(1,0){0.21}}
\multiput(34.96,40.41)(0.22,0.12){2}{\line(1,0){0.22}}
\multiput(35.4,40.65)(0.22,0.11){2}{\line(1,0){0.22}}
\multiput(35.84,40.88)(0.22,0.11){2}{\line(1,0){0.22}}
\multiput(36.28,41.1)(0.23,0.1){2}{\line(1,0){0.23}}
\multiput(36.73,41.31)(0.23,0.1){2}{\line(1,0){0.23}}
\multiput(37.19,41.51)(0.23,0.09){2}{\line(1,0){0.23}}
\multiput(37.65,41.69)(0.47,0.17){1}{\line(1,0){0.47}}
\multiput(38.12,41.86)(0.47,0.16){1}{\line(1,0){0.47}}
\multiput(38.59,42.02)(0.48,0.15){1}{\line(1,0){0.48}}
\multiput(39.07,42.17)(0.48,0.14){1}{\line(1,0){0.48}}
\multiput(39.55,42.31)(0.48,0.12){1}{\line(1,0){0.48}}
\multiput(40.03,42.43)(0.49,0.11){1}{\line(1,0){0.49}}
\multiput(40.52,42.55)(0.49,0.1){1}{\line(1,0){0.49}}
\multiput(41.01,42.65)(0.49,0.09){1}{\line(1,0){0.49}}
\multiput(41.5,42.73)(0.5,0.08){1}{\line(1,0){0.5}}
\multiput(42,42.81)(0.5,0.06){1}{\line(1,0){0.5}}
\multiput(42.5,42.87)(0.5,0.05){1}{\line(1,0){0.5}}
\multiput(43,42.92)(0.5,0.04){1}{\line(1,0){0.5}}
\multiput(43.5,42.96)(0.5,0.03){1}{\line(1,0){0.5}}
\multiput(44,42.99)(0.5,0.01){1}{\line(1,0){0.5}}
\put(44.5,43){\line(1,0){0.5}}
\multiput(45,43)(0.5,-0.01){1}{\line(1,0){0.5}}
\multiput(45.5,42.99)(0.5,-0.03){1}{\line(1,0){0.5}}
\multiput(46.01,42.96)(0.5,-0.04){1}{\line(1,0){0.5}}
\multiput(46.51,42.92)(0.5,-0.05){1}{\line(1,0){0.5}}
\multiput(47.01,42.87)(0.5,-0.06){1}{\line(1,0){0.5}}
\multiput(47.51,42.81)(0.5,-0.08){1}{\line(1,0){0.5}}
\multiput(48,42.73)(0.5,-0.09){1}{\line(1,0){0.5}}
\multiput(48.5,42.65)(0.49,-0.1){1}{\line(1,0){0.49}}
\multiput(48.99,42.55)(0.49,-0.11){1}{\line(1,0){0.49}}
\multiput(49.48,42.43)(0.49,-0.12){1}{\line(1,0){0.49}}
\multiput(49.97,42.31)(0.48,-0.14){1}{\line(1,0){0.48}}
\multiput(50.45,42.17)(0.48,-0.15){1}{\line(1,0){0.48}}
\multiput(50.93,42.02)(0.48,-0.16){1}{\line(1,0){0.48}}
\multiput(51.41,41.86)(0.47,-0.17){1}{\line(1,0){0.47}}
\multiput(51.88,41.69)(0.23,-0.09){2}{\line(1,0){0.23}}
\multiput(52.35,41.51)(0.23,-0.1){2}{\line(1,0){0.23}}
\multiput(52.81,41.31)(0.23,-0.1){2}{\line(1,0){0.23}}
\multiput(53.27,41.1)(0.23,-0.11){2}{\line(1,0){0.23}}
\multiput(53.72,40.88)(0.22,-0.11){2}{\line(1,0){0.22}}
\multiput(54.17,40.65)(0.22,-0.12){2}{\line(1,0){0.22}}
\multiput(54.61,40.41)(0.22,-0.13){2}{\line(1,0){0.22}}
\multiput(55.04,40.16)(0.21,-0.13){2}{\line(1,0){0.21}}
\multiput(55.47,39.9)(0.21,-0.14){2}{\line(1,0){0.21}}
\multiput(55.89,39.63)(0.21,-0.14){2}{\line(1,0){0.21}}
\multiput(56.3,39.34)(0.2,-0.15){2}{\line(1,0){0.2}}
\multiput(56.71,39.05)(0.13,-0.1){3}{\line(1,0){0.13}}
\multiput(57.1,38.75)(0.13,-0.1){3}{\line(1,0){0.13}}
\multiput(57.49,38.43)(0.13,-0.11){3}{\line(1,0){0.13}}
\multiput(57.87,38.11)(0.12,-0.11){3}{\line(1,0){0.12}}
\multiput(58.25,37.78)(0.12,-0.11){3}{\line(1,0){0.12}}
\multiput(58.61,37.44)(0.12,-0.12){3}{\line(1,0){0.12}}
\multiput(58.97,37.08)(0.12,-0.12){3}{\line(0,-1){0.12}}
\multiput(59.31,36.73)(0.11,-0.12){3}{\line(0,-1){0.12}}
\multiput(59.65,36.36)(0.11,-0.13){3}{\line(0,-1){0.13}}
\multiput(59.98,35.98)(0.11,-0.13){3}{\line(0,-1){0.13}}
\multiput(60.29,35.6)(0.1,-0.13){3}{\line(0,-1){0.13}}
\multiput(60.6,35.21)(0.15,-0.2){2}{\line(0,-1){0.2}}
\multiput(60.89,34.81)(0.14,-0.2){2}{\line(0,-1){0.2}}
\multiput(61.18,34.4)(0.14,-0.21){2}{\line(0,-1){0.21}}
\multiput(61.46,33.99)(0.13,-0.21){2}{\line(0,-1){0.21}}
\multiput(61.72,33.57)(0.13,-0.21){2}{\line(0,-1){0.21}}
\multiput(61.97,33.14)(0.12,-0.22){2}{\line(0,-1){0.22}}
\multiput(62.21,32.71)(0.12,-0.22){2}{\line(0,-1){0.22}}
\multiput(62.44,32.28)(0.11,-0.22){2}{\line(0,-1){0.22}}
\multiput(62.66,31.83)(0.1,-0.22){2}{\line(0,-1){0.22}}
\multiput(62.87,31.38)(0.1,-0.23){2}{\line(0,-1){0.23}}
\multiput(63.07,30.93)(0.09,-0.23){2}{\line(0,-1){0.23}}
\multiput(63.25,30.47)(0.17,-0.46){1}{\line(0,-1){0.46}}
\multiput(63.42,30.01)(0.16,-0.47){1}{\line(0,-1){0.47}}
\multiput(63.58,29.54)(0.15,-0.47){1}{\line(0,-1){0.47}}
\multiput(63.72,29.07)(0.13,-0.47){1}{\line(0,-1){0.47}}
\multiput(63.86,28.6)(0.12,-0.48){1}{\line(0,-1){0.48}}
\multiput(63.98,28.13)(0.11,-0.48){1}{\line(0,-1){0.48}}
\multiput(64.09,27.65)(0.1,-0.48){1}{\line(0,-1){0.48}}
\multiput(64.19,27.17)(0.08,-0.48){1}{\line(0,-1){0.48}}
\multiput(64.27,26.68)(0.07,-0.49){1}{\line(0,-1){0.49}}
\multiput(64.34,26.2)(0.06,-0.49){1}{\line(0,-1){0.49}}
\multiput(64.4,25.71)(0.05,-0.49){1}{\line(0,-1){0.49}}
\multiput(64.44,25.22)(0.03,-0.49){1}{\line(0,-1){0.49}}
\multiput(64.48,24.73)(0.02,-0.49){1}{\line(0,-1){0.49}}

\put(19,63){\makebox(0,0)[cc]{}}

\linethickness{0.3mm}
\put(52,24){\line(0,1){4}}
\put(52,28){\vector(0,1){0.12}}
\put(52,24){\vector(0,-1){0.12}}
\put(17,28){\makebox(0,0)[cc]{$Z_{2n}$}}

\put(16,20){\makebox(0,0)[cc]{$R^{-}{\bold{e}_1}$}}

\linethickness{0.3mm}
\multiput(5,24)(2,0){7}{\line(1,0){1}}
\put(39,16){\makebox(0,0)[cc]{}}

\put(45,25){\makebox(0,0)[cc]{}}

\linethickness{0.3mm}
\put(45,28){\line(0,1){4}}
\put(45,28){\vector(0,-1){0.12}}
\linethickness{0.3mm}
\put(21,28){\line(1,0){3}}
\put(24,28){\vector(1,0){0.12}}
\put(45,35){\makebox(0,0)[cc]{$A_{2n+1}$}}

\put(94,36){\makebox(0,0)[cc]{$B(0,q_N), \; N \text{ even }$}}

\linethickness{0.3mm}
\multiput(46,22)(0.52,-0.12){50}{\line(1,0){0.52}}
\put(46,22){\vector(-4,1){0.12}}
\put(88,15){\makebox(0,0)[cc]{$B(0,q_N), \; N \text{ odd }$}}

\linethickness{0.3mm}
\multiput(58,34)(0.65,0.12){17}{\line(1,0){0.65}}
\put(58,34){\vector(-4,-1){0.12}}
\linethickness{0.3mm}
\put(8,24){\line(0,1){19}}
\put(8,43){\vector(0,1){0.12}}
\put(8,24){\vector(0,-1){0.12}}
\put(4,34){\makebox(0,0)[cc]{$q_{2n}$}}

\linethickness{0.3mm}
\put(6,43){\line(1,0){4}}
\linethickness{0.3mm}
\put(45,24){\circle{6}}

\put(59,25){\makebox(0,0)[cc]{$\frac{4}{3}q_{2n+1}$}}

\end{picture}